\documentclass[a4paper, 11pt,reqno]{amsart}

\usepackage[T1]{fontenc}      
\usepackage[foot]{amsaddr}    

\usepackage[british]{babel}   
\usepackage{csquotes}         
\usepackage[babel]{microtype} 

\usepackage{mathtools}        
\usepackage{amssymb}          
\usepackage{amsthm}           
\usepackage{amstext}          
\usepackage{mleftright}       
\usepackage{gensymb}          
\usepackage[makeroom]{cancel} 
\usepackage{mathdots}         
\usepackage{mathrsfs}         

\usepackage{stickstootext}
\usepackage[stickstoo,varbb]{newtxmath} 
\makeatletter
\DeclareFontFamily{U}{ntxmia}{\skewchar \font =127}
 \DeclareFontShape{U}{ntxmia}{m}{it}{
                        <-> \ntxmath@scaled ntxmia
                      }{}    
                      \DeclareFontShape{U}{ntxmia}{b}{it}{
                        <-> \ntxmath@scaled ntxbmia
                      }{}
\makeatother

\usepackage{graphicx}         
\usepackage{psfrag}           
\usepackage{caption}          
\usepackage{tikz}             
\usetikzlibrary{backgrounds}

\usepackage{booktabs}         

\usepackage{enumitem}         

\usepackage[margin=1in]{geometry} 

\usepackage[textsize=footnotesize]{todonotes}

\usepackage[numbers,sort]{natbib}

\makeatletter
\def\NAT@spacechar{~}
\makeatother

\usepackage{hyperref}              
\usepackage[capitalise]{cleveref}  
\hypersetup{colorlinks=true,
   citecolor=blue,
   filecolor=blue,
   linkcolor=blue,
   urlcolor=blue
}
\usepackage{url}

\crefname{figure}{figure}{figures}
\Crefname{figure}{Figure}{Figures}

\usepackage{algorithm}
\usepackage{algpseudocode}

\allowdisplaybreaks

\newtheorem{definition}{Definition}[section]
\newtheorem{claim}{Claim}

\newtheorem{theorem}[definition]{Theorem}

\newtheorem{lemma}[definition]{Lemma}

\newtheorem{conjecture}[definition]{Conjecture}
\newtheorem{question}[definition]{Question}

\newtheorem{remark}[definition]{Remark}

\newenvironment{claimproof}{%
\let\origqed=\qedsymbol%
\renewcommand{\qedsymbol}{$\blacktriangleleft$}%
\begin{proof}}{\end{proof}\let\qedsymbol=\origqed}

\numberwithin{equation}{section}

\renewcommand{\binom}[2]{\ensuremath{\mleft(\kern-.1em\genfrac{}{}{0pt}{}{#1}{#2}\kern-.1em\mright)}}    
\newcommand{\inbinom}[2]{\ensuremath{\bigl(\kern-.1em\genfrac{}{}{0pt}{}{#1}{#2}\kern-.1em\bigr)}} 

\newcommand*\nume{\ensuremath{\mathrm{e}}}

\newcommand{\cC}{\mathcal{C}}

\makeatletter
\def\moverlay{\mathpalette\mov@rlay}
\def\mov@rlay#1#2{\leavevmode\vtop{%
  \baselineskip\z@skip \lineskiplimit-\maxdimen
  \ialign{\hfil$\m@th#1##$\hfil\cr#2\crcr}}}
\newcommand{\charfusion}[3][\mathord]{
    #1{\ifx#1\mathop\vphantom{#2}\fi
        \mathpalette\mov@rlay{#2\cr#3}
      }
    \ifx#1\mathop\expandafter\displaylimits\fi}
\makeatother
\newcommand{\cupdot}{\charfusion[\mathbin]{\cup}{\cdot}}

\newcommand{\bigO}{\ensuremath{\mathcal{O}}}

\newcommand{\dist}{\operatorname{dist}}

\newcommand{\COMMENT}[1]{}

\title{Hamiltonicity of graphs perturbed by a random regular graph}

\author[A.~Espuny D\'iaz]{Alberto Espuny D\'iaz}
\email{alberto.espuny-diaz@tu-ilmenau.de}
\author[A.~Gir\~{a}o]{Ant\'onio Gir\~{a}o}
\email{girao@maths.ox.ac.uk}

\address[Espuny D\'iaz]{Institut f\"ur Mathematik, Technische Universit\"at Ilmenau, 98684 Ilmenau, Germany.}
\address[Gir\~{a}o]{Mathematical Institute, University of Oxford, Oxford OX2 6GG, United Kingdom.}

\thanks{This project has received partial funding from the European Research Council (ERC) under the European Union's Horizon 2020 research and innovation programme (grant agreement no.~786198, A.~Espuny D\'iaz).
The research leading to these results was also partially supported by the Carl Zeiss Foundation (A.~Espuny D\'iaz) and the EPSRC grant no.\ EP/N019504/1 (A.~Gir\~ao) and EPSRC grant no.\ EP/V007327/1 (A.~Gir\~ao).}

\date{\today}

\begin{document}

\begin{abstract}
We study Hamiltonicity and pancyclicity in the graph obtained as the union of a deterministic $n$-vertex graph $H$ with $\delta(H)\geq\alpha n$ and a random $d$-regular graph $G$, for $d\in\{1,2\}$.
When $G$ is a random $2$-regular graph, we prove that a.a.s.~$H\cup G$ is pancyclic for all $\alpha\in(0,1]$, and also extend our result to a range of sublinear degrees.
When $G$ is a random $1$-regular graph, we prove that a.a.s.~$H\cup G$ is pancyclic for all $\alpha\in(\sqrt{2}-1,1]$, and this result is best possible.
Furthermore, we show that this bound on $\delta(H)$ is only needed when $H$ is `far' from containing a perfect matching, as otherwise we can show results analogous to those for random $2$-regular graphs.
Our proofs provide polynomial-time algorithms to find cycles of any length.
\end{abstract}

\maketitle

\section{Introduction}

Two classical areas of research in graph theory are that which deals with extremal results, and that which studies properties of random structures.
In the former area, one often looks for sufficient conditions to guarantee that a graph satisfies a certain property.
An example of such a result is the well-known Dirac's theorem~\cite{Dirac52}, which states that any graph $G$ on $n\geq3$ vertices with minimum degree at least $n/2$ contains a Hamilton cycle (that is, a cycle which contains all vertices of $G$).
In the latter, one considers a random structure, chosen according to some distribution, and studies whether it satisfies a given property with sufficiently high probability.
For instance, the binomial random graph $G_{n,p}$ (where $G_{n,p}$ is obtained by considering a vertex set of size~$n$ and adding each of the possible $\inbinom{n}{2}$ edges with probability $p$ and independently of each other) is known to contain a Hamilton cycle asymptotically almost surely (a.a.s.) whenever $p\geq(1+\epsilon)\log n/n$~\cite{Kor77}, while if $p\leq(1-\epsilon)\log n/n$, then a.a.s.~the graph is not even connected.

A more recent area of research at the interface of extremal combinatorics and random graph theory is that of \emph{randomly perturbed graphs}.
The general setting is as follows.
One considers an arbitrary graph $H$ from some class of graphs (usually, given by a minimum degree condition) and a random graph $G$, and asks whether the union of $H$ and $G$ satisfies a desired property a.a.s.
This can be seen as a way to bridge between the two areas described above.
Indeed, suppose that $H$ is any $n$-vertex graph with minimum degree $\delta(H)\geq\alpha n$ and $G=G_{n,p}$ is a binomial random graph (on the same vertex set as $H$), and consider the property of being Hamiltonian (that is, containing a Hamilton cycle).
If $\alpha\geq1/2$, then Dirac's theorem guarantees that $H\cup G$ is Hamiltonian, for all values of $p$.
If, on the other hand, $\alpha=0$, $H$ could be the empty graph, so we are left simply with the random graph $G$.
The relevant question, then, is whether, for all $\alpha\in(0,1/2)$, any graph $H$ with minimum degree $\alpha n$ is sufficiently `close' to Hamiltonicity that adding a random graph $G$ (which is itself not a.a.s.~Hamiltonian) will yield Hamiltonicity.
In particular, the goal is to determine whether the range of~$p$ for which this holds is significantly larger than the range for $G_{n,p}$ itself.

The randomly perturbed graph model was introduced by \citet{BFM03}, who studied precisely the problem of Hamiltonicity.
They proved that, for any constant $\alpha>0$, if $H$ is an $n$-vertex graph with $\delta(H)\geq\alpha n$, then a.a.s.~$H\cup G_{n,p}$ is Hamiltonian for all $p\geq C(\alpha)/n$.
Note, in particular, that this increases the range of $p$ given by the random graph model.
This result was very recently generalised by \citet{HMMMP21} to allow $\alpha$ to tend to $0$ with $n$ (that is, to consider graphs $H$ which are not dense), similarly improving the range of~$p$.
In a different direction, the results of \citet{BFM03} about Hamiltonicity were generalised by \citet{KKS16} to pancyclicity, that is, the property of containing cycles of all lengths between $3$ and $n$.

Other properties that have been studied in the context of randomly perturbed graphs are, e.g., the existence of powers of Hamilton cycles~\cite{BPSS22,BMPP20,DRRS20,ADRRS21,NT21}, $F$-factors~\cite{BTW19,HMT19+,BPSS20,BPSS21}, spanning bounded degree trees~\cite{KKS17,BHKMPP19} and (almost) unbounded degree trees~\cite{JK19}, or general bounded degree spanning graphs~\cite{BMPP20}.
The model of randomly perturbed graphs has also been extended to other settings.
For instance, Hamiltonicity has been studied in randomly perturbed directed graphs~\cite{BFM03,KKS16}\COMMENT{Both Hamiltonicity and pancyclicity.}, hypergraphs~\cite{KKS16,MM18,BHKM19,HZ20,CHT20}\COMMENT{\cite{BHKM19,CHT20} are about powers of Hamilton cycles. \cite{HZ20} is the only one with degree conditions for $H$ (best possible); all the others have codegree conditions.} and subgraphs of the hypercube~\cite{CEGKO20}.
A common phenomenon in randomly perturbed graphs is that, by considering the union with a dense graph (i.e., with linear degrees), the threshold for the probabilities of different properties is significantly lower than that of the classical $G_{n,p}$ model.

In this paper, we study the analogous problem where the random graph that is added to a deterministic graph is a random \emph{regular} graph (a graph is \emph{regular} when all its vertices have the same degree).
To be more precise, we will consider a given $n$-vertex graph $H$ with $\delta(H)\geq\alpha n$, and we will study the Hamiltonicity and pancyclicity of $H\cup G$, where $G$ is a random $d$-regular graph, for some $d\in\mathbb{N}$.
We write $G_{n,d}$ for a graph chosen uniformly at random from the set of all $d$-regular graphs on~$n$ vertices (we implicitly assume that $nd$ is even throughout).
The model of random regular graphs, though harder to analyse than the binomial model, has been studied thoroughly.
In particular, it is a well-known fact that $G_{n,d}$ is a.a.s.~Hamiltonian for all $3\leq d\leq n-1$~\cite{RW92,RW94,CFR02,KSVW01}.
It follows that, for all $d\geq3$, $H\cup G_{n,d}$ is a.a.s.~Hamiltonian independently of $H$, so the problem we consider only becomes relevant for $d\in\{1,2\}$.
For more information about random regular graphs, we recommend the survey of \citet{Wor99}.

It turns out that the behaviour of the problem is quite different in each of these two cases.
When we consider $G=G_{n,2}$ (also referred to as a random $2$-factor), we obtain a result similar to those obtained in the binomial setting: for every constant $\alpha>0$, if $H$ is an $n$-vertex graph with $\delta(H)\geq\alpha n$, then a.a.s.~$H\cup G$ is Hamiltonian.
Even more, we can prove that, in this case, a.a.s. $H\cup G$ is pancyclic, and also extend the range of $\alpha$ for which this holds to some function of $n$ which tends to $0$.

\begin{theorem}\label{thm:2factor}
Let $\alpha=\omega((\log n/n)^{1/4})$\COMMENT{We could actually compute an explicit constant $C$ and write $\alpha\geq C(\log n/n)^{1/4}$.}, and let $H$ be an $n$-vertex graph with $\delta(H)\geq \alpha n$.
Then, a.a.s. $H\cup G_{n,2}$ is pancyclic. 
\end{theorem}

In contrast to this, when we consider $G=G_{n,1}$ (that is, a random perfect matching), there is a particular value $\alpha^*<1/2$ such that, for all $\alpha>\alpha^*$, if $H$ is an $n$-vertex graph with $\delta(H)\geq\alpha n$, then a.a.s.\ $H\cup G$ is Hamiltonian.

\begin{theorem}\label{thm:matching}
For all $\epsilon>0$, the following holds.
Let $\alpha\coloneqq(1+\epsilon)(\sqrt{2}-1)$, and let $H$ be an $n$-vertex graph with $\delta(H)\geq \alpha n$.
Then, a.a.s.~$H\cup G_{n,1}$ is pancyclic. 
\end{theorem}

This result is best possible.
Indeed, for every $\alpha<\sqrt{2}-1$, there exist graphs $H$ with $\delta(H)\geq\alpha n$ such that $H\cup G_{n,1}$ is not a.a.s.~Hamiltonian.
As we will discuss later (see \cref{sect:1factor}), the main extremal construction of $H$ for the lower bound is a complete unbalanced bipartite graph.
One key feature of this example is that $H$ does not contain a very large matching.
Indeed, when we further impose that $H$ contains an (almost) perfect matching, we can obtain the following result analogous to \cref{thm:2factor}.

\begin{theorem}\label{thm:matching2}
Let $\alpha=\omega((\log n/n)^{1/4})$, and let $H$ be an $n$-vertex graph with $\delta(H)\geq \alpha n$ which contains a matching $M$ which covers $n-o(\alpha^2n)$\COMMENT{Again, constant times $\alpha^2n$ would be enough.} vertices.
Then, a.a.s.~$H\cup G_{n,1}$ is pancyclic. 
\end{theorem}

The rest of the paper is organised as follows.
In \cref{sect:prelim} we present our notation and some preliminary results which will be useful for us, and in \cref{sect:aux} we prove some properties of $G_{n,1}$ and $G_{n,2}$ regarding their edge distribution and component structure which are crucial for our proofs.
We devote \cref{sect:2factor} to proving \cref{thm:2factor,thm:matching2}, and defer the proof of \cref{thm:matching} to \cref{sect:1factor}.
Finally, we discuss some possible extensions in \cref{sect:concl}.


\section{Preliminaries}\label{sect:prelim}

\subsection{Notation}

For any $n\in\mathbb{Z}$, we denote $[n]\coloneqq\{i\in\mathbb{Z}:1\leq i\leq n\}$ and $[n]_0\coloneqq\{i\in\mathbb{Z}:0\leq i\leq n\}$.
In particular, $[0]=\varnothing$.
Given any $a,b,c\in\mathbb{R}$, we write $c=a\pm b$ if $c\in[a-b,a+b]$.
Whenever we use a hierarchy, the constants are chosen from right to left. 
To be more precise, when claiming that a statement holds for $0<a\ll b\leq1$, we mean that it holds for all $0<b\leq1$ and all $0<a\leq f(b)$, for some unspecified non-decreasing function $f$.
For our asymptotic statements we often use the standard $\bigO$ notation; when doing so, we always understand that the functions are non-negative.

Throughout this paper, the word \emph{graph} will refer to simple, undirected graphs.
Whenever we consider graphs with loops or parallel edges, we will call them \emph{multigraphs}.
Whenever we consider a (multi)graph $G$ on $n$ vertices, we implicitly assume that $V(G)=[n]$ (in particular, all graphs are labelled).

Given any (multi)graph $G=(V,E)$ and any two (not necessarily disjoint) sets $A,B\subseteq V$, we denote the (multi)set of edges of $G$ with both endpoints in $A$ by $E_G(A)$, and the (multi)set of edges with one endpoint in $A$ and one in $B$ by $E_G(A,B)$.
If $A=\{a\}$ is a singleton, for simplicity we will write $E_G(a,B)\coloneqq E_G(\{a\},B)$, and similarly in the rest of the notation.
We denote $e_G(A)\coloneqq|E_G(A)|$ and $e_G(A,B)\coloneqq|E_G(A,B)|$.
We write $G[A]\coloneqq(A,E_G(A))$ for the subgraph of $G$ induced by $A$.
If~$A$ and~$B$ are disjoint, we write $G[A,B]\coloneqq(A\cup B,E_G(A,B))$ for the bipartite subgraph of $G$ induced by $A$ and $B$.
Given two (multi)graphs $G$ and $H$, we write $G\cup H\coloneqq(V(G)\cup V(H),E(G)\cup E(H))$, and if $V(H)\subseteq V(G)$, we write $G\setminus H\coloneqq(V(G),E(G)\setminus E(H))$.
We denote $G-A\coloneqq G[V\setminus A]$.
We will write $A\triangle B\coloneqq (A\cup B)\setminus (A\cap B)$ for the symmetric difference of $A$ and $B$.
We will often abbreviate edges $e=\{x,y\}$ as $e=xy$; recall, however, that edges are sets of vertices, and will be used as such throughout.
Given any vertex $x\in V$, we let $d_G(x)\coloneqq|\{e\in E:x\in e\}|+|\{e\in E:e=xx\}|$ denote its \emph{degree} in $G$ (that is, loops count twice towards the degree of a vertex).
We write $\Delta(G)$ and $\delta(G)$ for the maximum and minimum vertex degrees in $G$, respectively.
We write $N_G(x)$ for the neighbourhood of $x$ in $G$, that is, the set of vertices $y\in V$ such that $xy\in E$.
Given any set $A\subseteq V$, we denote $N_G(A)\coloneqq\bigcup_{x\in A}N_G(x)$.
Given any set of edges $E'$, we write $V(E')\coloneqq\bigcup_{e\in E'}e$.

A \emph{path} $P$ can be seen as a set of vertices which can be labelled in such a way that there is an edge between any two consecutive vertices.
Equivalently, they can be defined by the corresponding set of edges.
We will often consider isolated vertices as a degenerate case of paths.
If the \emph{endpoints} of $P$ (that is, the first and last vertices, according to the labelling) are $x$ and $y$, we often refer to $P$ as an $(x,y)$-path.
We refer to all vertices of a path $P$ which are not its endpoints as \emph{internal} vertices.
The \emph{length} of a path is equal to the number of edges it contains (the same definition holds for cycles).
The \emph{distance} between two vertices $x,y$ in a (multi)graph $G$, denoted by $\dist_G(x,y)$, is the length of the shortest $(x,y)$-path $P\subseteq G$ (if there is no such path, the distance is said to be infinite).
Given any two sets of vertices $A,B\subseteq V(G)$, we define $\dist_G(A,B)\coloneqq\min_{x\in A,y\in B}\dist_G(x,y)$.

Whenever we consider asymptotic statements, we will use the convention of abbreviating a sequence of graphs $\{G_i\}_{i\in\mathbb{N}}$ such that $|V(G_i)|\to\infty$ by simply writing that $G$ is an $n$-vertex graph (and implicitly understanding that $n\to\infty$).
In this setting, we say that an event $\mathcal{E}$ holds \emph{asymptotically almost surely} (a.a.s.~for short) if $\mathbb{P}[\mathcal{E}]=1-o(1)$.
We will ignore rounding issues whenever this does not affect the arguments.

\subsection{The configuration model}

In order to study random regular graphs, one of the most useful tools is the \emph{configuration model}, first introduced by \citet{Bol80}, which gives a process that samples $d$-regular graphs uniformly at random.
In its simplest form, the process works as follows.
For each $i\in[n]$, consider a set of $d$ vertices $x_{i,1},\ldots,x_{i,d}$.
Then, choose a uniformly random perfect matching $M$ covering the set $\{x_{i,j}:i\in[n],j\in[d]\}$.
Now, generate a multigraph $G=G(M)=([n],E)$ where, for each edge $x_{i,j}x_{i',j'}\in M$, an edge $ii'$ is added to $E$ (if $i=i'$, this creates a loop).
For any $i\in[n]$, we will refer to $\{x_{i,j}:j\in[d]\}$ as the \emph{extended set} of~$i$.
Similarly, for any $A\subseteq[n]$, we will refer to $\{x_{i,j}:i\in A,j\in[d]\}$ as the \emph{extended set} of $A$.

Whenever we produce a (multi)graph $G$ following the process above, we will write $G\sim\mathcal{C}_{n,d}$.
We will write $M\sim\mathcal{C}^*_{n,d}$ for the perfect matching $M$ such that $G=G(M)$.
In order to easily distinguish both models, we will refer to $M\sim\mathcal{C}^*_{n,d}$ as a \emph{configuration}.
By abusing notation, we will also sometimes write $\mathcal{C}^*_{n,d}$ to denote the set of all configurations with parameters $n$ and $d$.
Observe that, in particular, for every $n$-vertex $d$-regular multigraph $G$, there exists at least one configuration $M\in\mathcal{C}_{n,d}^*$ such that $G=G(M)$.

Given any two configurations $M,M'\in\mathcal{C}^*_{n,d}$, we write $M\sim M'$ if there exist $u_1u_2, v_1v_2 \in M$ such that $M' = (M\setminus \{u_1u_2, v_1v_2\}) \cup \{u_1v_1, u_2v_2\}$.
The following lemma bounds the probability that certain variables on configurations deviate from their expectation (see, e.g.,~\cite[Lemma~4.3]{CEGKO19}).

\begin{lemma}\label{lem:martingale}
Let $d\in\mathbb{N}$ be fixed.
Let $c>0$ and let $X$ be a random variable on $\mathcal{C}^*_{n,d}$ such that, for every pair of configurations $M\sim M'$, we have $|X(M)-X(M')| \le c$.
Then, for all $t\in\mathbb{N}$, 
\[\mathbb{P}[|X - \mathbb{E}[X]| \ge t] \le 2\nume^{- \frac{t^2}{2ndc^2}}.\]
\end{lemma}

Whenever sampling $G\sim\mathcal{C}_{n,d}$, the process described above may yield a multigraph with both loops and parallel edges.
(Of course, this does not apply to the case $d=1$, where we simply generate a uniformly random perfect matching.)
However, upon conditioning on the event that $G$ is a simple graph, such a graph is a uniformly random $d$-regular graph.
The following is by now a well-established fact (see, e.g.,~\cite[section~2.2]{Wor99}):

\begin{lemma}\label{lem:simple}
Let $d\in\mathbb{N}$ be fixed, and let $G\sim\mathcal{C}_{n,d}$.
If $n$ is sufficiently large, then 
\[\mathbb{P}[G\text{ is simple}]\geq \nume^{-d^2/4}.\]
\end{lemma}


\section{Properties of random regular graphs of low degree}\label{sect:aux}

We will need to have bounds on the number of components of random $2$-factors (note that every $2$-regular graph is a union of vertex-disjoint cycles), as well as the number of cycles in the union of a (not necessarily spanning) matching and a random perfect matching.

\begin{lemma}\label{lem:fewcycles}
The following statements hold.
\begin{enumerate}[label=$(\mathrm{\roman*})$]
    \item\label{item:fewcycles1} A.a.s.~the number of components of $G_{n,2}$ is at most $\log^2 n$.
    \item\label{item:fewcycles2} Let $M$ be any (not necessarily perfect) matching on $[n]$.
    Then, a.a.s.~$M\cup G_{n,1}$ contains at most $\log^2n$ cycles, and its number of components is at most $n/2-|M|+\log^2n$.
\end{enumerate}
\end{lemma}

\begin{proof}
We are going to generate a multigraph $G\sim\mathcal{C}_{n,2}$ according to the configuration model (for the proof of \ref{item:fewcycles2}, this will be conditioned on certain edges being present).
For this, it is important to note that a uniformly random perfect matching on any set $A$ of $2n$ vertices (that is, a configuration $M\sim\mathcal{C}_{n,2}^*$) can be obtained iteratively by choosing $n$ edges in $n$ steps.
For step~$i$, let $B_i$ be the set of all vertices which are covered by the edges of the first $i-1$ steps.
Then, choose any vertex $x_i\in A\setminus B_i$, and choose a neighbour $y_i\in A\setminus(B_i\cup\{x_i\})$ for $x_i$ uniformly at random.
The choice of $x_i$ in each step of this process can be made arbitrarily.
If we generate a random perfect matching on $2n$ vertices conditioned on $m$ given pairs being present in this matching, we can follow the process above starting at step $m+1$, assuming that the given $m$ edges were chosen in the first $m$ steps of the process.
For our proofs, we are going to generate a uniformly random configuration following this process.

For each $j\in[n]$, let $A_j$ be the extended set of $j$, and let $A\coloneqq\bigcup_{j\in[n]}A_j$.
Consider first the proof of \ref{item:fewcycles1}.
In the $i$-th step, the choice of the vertex $x_i$ will be made as follows:
if $i\geq2$ and for the unique $j\in[n]$ such that $y_{i-1}\in A_j$ we have that $|A_j\cap B_i|=1$, then let $x_i$ be the unique vertex in $A_j\setminus\{y_{i-1}\}$;
otherwise\COMMENT{for all $j\in[n]$ we have that $|A_j\cap B_i|$ is even}, let $x_i\in A\setminus B_i$ be arbitrary.
Roughly speaking, this choice of $x_i$ at each step means that we are revealing the edges of $G\sim\mathcal{C}_{n,2}$ in such a way that we reveal all edges of each connected component of $G$ before moving on to the next one.
In particular, we only start a new component at step $i+1$ when the randomly chosen vertex $y_{i}$ lies in the unique $A_j$ where exactly one vertex had already been picked\COMMENT{Note this vertex might be $x_{i}$, this has already been picked when $y_{i}$ is.}.
Whenever this happens, we say that the process has \emph{finished} a component at step $i$.

For each $i\in[n]$, let $X_i$ be an indicator random variable which takes value $1$ whenever the process finishes a component at step $i$, and $0$ otherwise.
Observe that, since $y_i$ is chosen uniformly at random, for each $i\in[n]$ we have that\COMMENT{Even more, the variables are independent, so if we wanted to obtain a tighter result we could use Chernoff bounds.}
\[\mathbb{P}[X_i=1]=\frac{1}{2n-2i+1}.\]
The total number of components $X$ of $G$ is then given by the sum of these indicator variables.
In particular, for $n$ sufficiently large it follows that\COMMENT{We have
\[\mathbb{E}[X]=\sum_{i=1}^n\frac{1}{2n-2i+1}=\sum_{i=1}^n\frac{1}{2i-1}\leq1+\sum_{i=2}^{2n}\frac{1}{i}\leq1+\log2+\log n\leq 2\log n.\]}
\[\mathbb{E}[X]=\sum_{i=1}^n\frac{1}{2n-2i+1}\leq2\log n.\]
The claim now follows immediately by Markov's inequality\COMMENT{We have that
\[\mathbb{P}[X\geq\log^2n]\leq\frac{\mathbb{E}[X]}{\log^2n}\leq\frac{2\log n}{\log^2n}=o(1).\]} and \cref{lem:simple}.

Consider now the proof of \ref{item:fewcycles2}.
First, extend the given matching $M$ into a perfect matching $M'$ arbitrarily.
For each $j\in[n]$, let $\sigma(j)$ be the unique $k\in[n]$ such that $\{j,k\}\in M'$.
Now, consider the extended set $A$ and, for each $\{j,k\}\in M'$, add the pair $x_{j,1}x_{k,1}$ to a matching $\tilde{M}$ on $A$, so we have that $G(\tilde{M})=M'$.
We are going to obtain a configuration $M''\sim\mathcal{C}_{n,2}^*$ conditioned on containing this matching $\tilde{M}$.
Observe that, since exactly one point in the extended set of each vertex has been covered by $\tilde{M}$, a random matching on the uncovered points corresponds to a random perfect matching on~$[n]$.
In particular, $G(M'')$ may contain some parallel edges (which correspond to `isolated' edges in $M'\cup G_{n,1}$), but it cannot contain any loops.

For each $i\in[n]\setminus[n/2]$, the choice of the vertex $x_i$ will be made as follows: if $i\geq n/2+2$ and for the unique $j\in[n]$ such that $y_{i-1}\in A_j$ we have that $|A_{\sigma(j)}\cap B_i|=1$, let $x_i$ be the unique vertex in $A_{\sigma(j)}\setminus B_i$; otherwise, choose $x_i$ arbitrarily.
Similarly to the proof of \ref{item:fewcycles1}, this allows us to count the number of components of $G(M'')$ (and, thus, of $M'\cup G_{n,1}$).
Note that we only start a new component at step $i+1$ if the randomly chosen $y_{i}$ lies in the unique (not completely covered) $A_j$ such that $A_{\sigma(j)}$ was covered before the choice of $y_{i}$.
Whenever this happens, we say that the process has finished a component at step $i$.

For each $i\in[n]\setminus[n/2]$, let $Y_i$ be a random variable which takes value $1$ whenever the process finishes a component at step $i$, and $0$ otherwise, and let $Y$ be the number of components of $G(M'')$.
As before, since $y_i$ is chosen uniformly at random, for each $i\in[n]\setminus[n/2]$ we have that
\[\mathbb{P}[Y_i=1]=\frac{1}{2n-2i+1},\]
and it follows that a.a.s.~$Y\leq\log^2n$.
Each of these components is either a cycle or an isolated pair of parallel edges.
Finally, $M\cup G_{n,1}$ can be obtained by contracting parallel edges into a single edge and then removing the edges of $M'\setminus(M\cup G_{n,1})$ from $G(M'')$, and this never creates any new cycles, so the bound on the number of cycles follows.

In order to bound the number of components of $M\cup G_{n,1}$, note that the deletion of each edge of $M'\setminus(M\cup G_{n,1})$ may create at most one new component.
Thus, a.a.s.~the number of components of $M\cup G_{n,1}$ is at most $\log^2n+(n/2-|M|)$.
\end{proof}

\begin{remark}\label{rem:fewcycles}
The proof of \cref{lem:fewcycles}\ref{item:fewcycles2} gives the following simple bound:
if $M$ is any (not necessarily perfect) matching on $[n]$, then a.a.s.~$|E(M)\cap E(G_{n,1})|\leq\log^2n$.
\end{remark}

The following edge-distribution property of $G_{n,1}$ and $G_{n,2}$ will also be useful for us.

\begin{lemma}\label{lem:triangsand3paths}
Let $\epsilon>0$ and $k=k(n)\leq n^3$\COMMENT{Actually true for any polynomial, cubic is enough for our proofs.}.
For each $i\in[k]$, let $\alpha_i=\alpha_i(n)$ and $\beta_i=\beta_i(n)$ be such that $\alpha_i\beta_i=\omega((\log n/n)^{1/2})$ and let $A_i,B_i\subseteq[n]$ be two not necessarily disjoint sets of vertices with $|A_i|\geq\alpha_in$ and $|B_i|\geq\beta_in$.
Let $G=G_{n,1}$ or $G=G_{n,2}$.
Then, a.a.s., for every $i\in[k]$, $A_i$ contains at least $(1-\epsilon)\alpha_i\beta_in$ vertices $z$ such that $N_G(z)\cap B_i\neq\varnothing$.
\end{lemma}

\begin{proof}
We will show how to prove the statement when $G=G_{n,2}$; the other case can be shown in exactly the same way, but avoiding the reference to \cref{lem:simple}.

Let $G'\sim\mathcal{C}_{n,2}$.
Let $\mathcal{E}$ be the event that the statement holds for $G'$.
First, fix some $i\in[k]$, let $Z\coloneqq\{z\in A_i:N_{G'}(z)\cap B_i\neq\varnothing\}$, and let $X\coloneqq|Z|$.
By using the configuration model, for each $z\in A_i$ and $n$ sufficiently large we have that\COMMENT{Consider the configuration model, i.e., we generate a uniformly random perfect matching on the extended set of $[n]$.
Each vertex is extended into two for the configuration model.
Let $z_1$ and $z_2$ conform the extended set of $z$.
The event that we are interested in is the event that $z_1$ or $z_2$ are joined to vertices of the extended set of $B_i$, which has size $2|B_i|$.
The probability that $z_1$ satisfies this is at least $(2|B_i|-1)/(2n-1)\geq(2|B_i|-1)/(2n)$.
Of course, if this does not hold for $z_1$, then it can still hold for $z_2$, so the probability is even larger (roughly twice as large). 
However, this is the best we can do for $G_{n,1}$ (in this case, the extended set of $B_i$ has size $|B_i|$ and the bound we get on the probability is $(|B_i|-1)/(n-1)\geq(|B_i|-1)/n\geq(1-\epsilon/2)\beta_i$), so we just use this bound here.} 
\[\mathbb{P}[z\in Z]\geq\frac{2|B_i|-1}{2n-1}\geq(1-\epsilon/2)\beta_i,\]
which implies that $\mathbb{E}[X]\geq(1-\epsilon/2)\alpha_i\beta_in$.

Now observe that any random variable on $d$-regular multigraphs obtained according to the configuration model can also be seen as a random variable on uniformly random configurations.
In particular, since any pair of configurations $M\sim M'$ are equal except for their edges at four vertices, it follows that, given any two configurations $M\sim M'$, we have $|X(M)-X(M')|\leq4$\COMMENT{This is just the trivial bound: we are counting edges and, since $M\sim M'$, only four edges change between one configuration and the other. We could be more careful and prove a better bound, but we do not care.}.
Thus, by \cref{lem:martingale} we conclude that $\mathbb{P}[X<(1-\epsilon)\alpha_i\beta_in]\leq \nume^{-\Omega(\alpha_i^2\beta_i^2n)}$\COMMENT{We have that
\[\mathbb{P}[X<(1-\epsilon)\alpha_i\beta_in]\leq\mathbb{P}[|X-\mathbb{E}[X]|\geq(\epsilon/2)(1-\epsilon/2)\alpha_i\beta_in]\leq2\nume^{-\frac{(\epsilon/2)^2(1-\epsilon/2)^2\alpha_i^2\beta_i^2n^2/16}{2n\cdot2\cdot16}}=2\nume^{-(\epsilon/2)^2(1-\epsilon/2)^2\alpha_i^2\beta_i^2n/2^{10}}.\]}.

By a union bound over all $i\in[k]$ and the bound on $\alpha_i\beta_i$ in the statement, we conclude that
\[\mathbb{P}[\mathcal{E}]\geq1-\sum_{i=1}^k\nume^{-\Omega(\alpha_i^2\beta_i^2n)}=1-o(1).\]
Finally, by applying \cref{lem:simple}, we have that the statement holds a.a.s.~for $G=G_{n,2}$.
\end{proof}


\section{Randomly perturbing graphs by a \texorpdfstring{$2$}{2}-factor}\label{sect:2factor}

We can now prove \cref{thm:2factor,thm:matching2} simultaneously.
We note that we believe the bound on $\delta(H)$ in the statements is far from optimal and, thus, we make no effort to optimise the constants throughout.

We will let $G=G_{n,2}$ or $G=G_{n,1}$ and want to show that $H\cup G$ is pancyclic.
Our general strategy will be to first construct a special path $P$ which can be used to `absorb' vertices into a given cycle.
To be more precise, we will set aside an arbitrary set of vertices of a suitable size and then construct a path $P$ which avoids this set of vertices.
Furthermore, we will ensure that each of the vertices set aside forms a triangle with one of the edges of $P$ (and each of the vertices does this with a different edge).
Then, by replacing the corresponding edge of $P$ by the path of length $2$ that forms the triangle, we can incorporate each of the vertices set aside into the path (thus `absorbing' the vertex). 
This will allow us to have some control over the length of the cycles which we can produce.
Once we have constructed the special path $P$, we will show that we can find an `almost' spanning cycle $C$ with $P\subseteq C$ which avoids the vertices that can be absorbed with $P$. 
We can then use the cycle $C$, together with the set of vertices that we can absorb into $P$, to find cycles of all lengths larger than that of $C$.
A similar strategy using $P$ will allow us to find all shorter cycles.

\begin{proof}[Proof of \cref{thm:2factor,thm:matching2}]
Let $\alpha=\omega((\log n/n)^{1/4})$, and let $G=G_{n,2}$ or $G=G_{n,1}$, depending on which of the statements we want to prove.
Condition on the event that the statements of \cref{lem:fewcycles,lem:triangsand3paths} hold, which occurs a.a.s., where we apply \cref{lem:triangsand3paths} to the pairs of sets $N_H(x)$ and $N_H(y)$ for each pair $(x,y)\in V(H)\times V(H)$ with $\epsilon=1/2$.
Thus, 
\begin{equation}\label{equa:3pathsandtriangles}
\begin{minipage}[c]{0.83\textwidth}
for all $(x,y)\in V(H)\times V(H)$ we have that $N_H(x)$ contains at least $\alpha^2n/2$ vertices $z$ such that $N_G(z)\cap N_H(y)\neq\varnothing$.
\end{minipage}\ignorespacesafterend 
\end{equation}
Whenever necessary, for our claims we assume that $n$ is sufficiently large.

For each pair of (not necessarily distinct) vertices $(x,y)\in V(H)\times V(H)$, consider the set
\[Z(x,y)\coloneqq\{zz'\in E(G):z\in N_H(x),z'\in N_H(y)\}.\]
This is a set of `available' edges for $(x,y)$\COMMENT{Note there can be some degenerate cases with $z=y$ or $z'=x$ (or maybe even both), but these are constantly many out of the linearly many, and should work fine with our notation as well.}.
Throughout the proof, we will use these edges to make alterations on graphs; every time we do so, we will update this set of available edges (in particular, we will always restrict the list to edges of a `current' graph; this will become clear later in the proof).
For simplicity of notation, whenever we consider an edge $zz'$ from some set $Z(x,y)$ (or any of their updated versions), we will implicitly assume that $z\in N_H(x)$ and $z'\in N_H(y)$; note, however, that these edges are not really `oriented' in the definition, and so $zz'\in Z(y,x)$ too.
By \eqref{equa:3pathsandtriangles}, for all $(x,y)\in V(H)\times V(H)$ we have 
\begin{equation}\label{equa:Zsize}
    |Z(x,y)|\geq\alpha^2n/4.
\end{equation}
Note further that, since $Z(x,y)\subseteq G$ (by abusing notation and identifying $Z(x,y)$ with a graph),
\begin{equation}\label{equa:degZ}
\Delta(Z(x,y))\leq2.
\end{equation} 

Take an arbitrary set $U\subseteq V(H)$ of $m\coloneqq\alpha^2n/1000$ vertices and label them as $u_1,\ldots,u_m$.
For each $j\in[m]$, iteratively, choose an edge $e_j=z_jz_j'\in Z(u_j,u_j)$ such that $e_j\cap U=\varnothing$ and $e_j\cap\bigcup_{k\in[j-1]}e_k=\varnothing$ (the existence of such edges is guaranteed in every step by \eqref{equa:Zsize} and \eqref{equa:degZ}).
Let $W\coloneqq\bigcup_{j\in[m]}e_j$\COMMENT{Note this is a set of vertices.}.
Now, for each $(x,y)\in V(H)\times V(H)$, we update the list of `available' edges for $(x,y)$ by letting 
\[Z'(x,y)\coloneqq\{zz'\in E(G):z\in N_H(x)\setminus(U\cup W),z'\in N_H(y)\setminus(U\cup W)\},\]
so it follows from \eqref{equa:Zsize} and \eqref{equa:degZ} and since $|U\cup W|=3m$ that\COMMENT{We could do better as follows. We count the number of `forbidden' edges. First, each vertex of $U$ forbids at most two edges. Then, each edge $e_j$ forbids at most three edges: itself and at most one more incident to each of its endpoints. Overall, this gives $5m=\alpha^2n/200$.
What is written below is a bit more careless: just the number of deleted vertices ($3m$) times at most $2$ edges forbidden for each vertex.}
\begin{equation}\label{equa:Zsize2}
    |Z'(x,y)|\geq\alpha^2n/4-3\alpha^2n/500\geq\alpha^2n/5.
\end{equation}
Now, for each $j\in[m-1]$, iteratively, choose an edge $f_j=w_jw_j'\in Z'(z_j',z_{j+1})$ satisfying that $f_j\cap\bigcup_{k\in[j-1]}f_k=\varnothing$.
The existence of such edges in each step follows by \eqref{equa:Zsize2} and \eqref{equa:degZ}.

Consider the path $P\coloneqq z_1z_1'w_1w_1'z_2\ldots w_{m-1}'z_mz_m'$ (the fact that this is a path follows by the previous choices of edges).
Note $|V(P)|\leq4m$.
Let $W'\coloneqq V(P)\setminus\{z_1,z_m'\}$.
Consider now the graph $G_0\coloneqq(G-(W'\cup U))\cup P$ or $G_0\coloneqq((M\cup G)-(W'\cup U))\cup P$\COMMENT{This seemingly weird definition is a short way to ensure that we delete all edges of the cycles incident to $P$ which do not lie in $P$; this is important so that we can later consider that we have a set of paths.}, depending on the statement we are trying to prove.
By \cref{lem:fewcycles}, this is a union of at most $\log^2n$ cycles and at most $(1+o(1))\alpha^2n/200$ paths, all vertex-disjoint (where we might have degenerate cases where some paths have length~$0$, that is, they are an isolated vertex).
Indeed, note that \cref{lem:fewcycles} asserts that the number of components of $G_{n,2}$ or $M\cup G_{n,1}$ is $o(\alpha^2n)$.
The removal of each of the vertices in $W'\cup U$ may increase the number of components by at most one, and $|W'\cup U|\leq5m=\alpha^2n/200$, so the claim follows.
Furthermore, after adding the path $P$ again, we have that $\Delta(G_0)\leq2$.
Indeed, the fact that the endpoints of $P$ have degree at most $2$ in $G_0$ follows since the first and last edges of $P$ lie in $E(G)$.

We will now use the edges of $H$ to iteratively combine these paths and cycles into a single cycle $C\subseteq(H\cup G)-U$ of length $n-m$ with $P\subseteq C$, which we will later use to obtain a Hamilton cycle (and, indeed, cycles of all lengths).
Let $V\coloneqq V(H)\setminus U$.
For each $(x,y)\in V^2$, we update the set of available edges by setting 
\[Z_0(x,y)\coloneqq\{zz'\in E(G):z\in N_H(x)\setminus(V(P)\cup U),z'\in N_H(y)\setminus(V(P)\cup U)\}.\]
It follows by \eqref{equa:Zsize} and \eqref{equa:degZ} and since $|V(P)\cup U|\leq5m$ that\COMMENT{Again, we could improve the constants here, but an easy bound is that we remove $5m$ vertices, each of which removes at most $2$ edges.}
\begin{equation}\label{equa:Zsize3}
|Z_0(x,y)|\geq\alpha^2n/4-\alpha^2n/100=6\alpha^2n/25.
\end{equation}

Throughout the upcoming process, we will define graphs $G_1,\ldots,G_t$, for some $t\in\mathbb{N}$, where each one is obtained from the previous by some `small' alteration.
All these graphs will be unions of vertex-disjoint paths and cycles spanning $V$.
The process will end when $G_t=C$, that is, when we obtain the desired cycle spanning $V$.
For each $i\in[t]$, the alteration that we perform on $G_{i-1}$ to construct $G_i$ will depend on its structure.
We present here a sketch; the full details are given in cases~\hyperlink{case1}{1} to \hyperlink{case3}{3} below.
While the current graph $G_{i-1}$ contains at least two paths (case~\hyperlink{case1}{1}), we create a new graph~$G_i$ with one fewer path and whose number of components does not increase by more than one.
If~$G_{i-1}$ contains exactly one path (case~\hyperlink{case2}{2}), then we either decrease the number of components or make sure that~$G_i$ has no paths (while not increasing the number of components).
Finally, if~$G_{i-1}$ contains at least two cycles and no paths (case~\hyperlink{case3}{3}), then we create a graph $G_i$ with one fewer component, but we possibly create a path in the process.
It follows from this description that the process must really end (assuming it can be carried out, which we prove later).
Indeed, we only apply case~\hyperlink{case1}{1} while at least two components are paths, and each time reduce the number of paths and do not increase the number of components by more than one.
Since $G_0$ contains at most $\log^2n$ cycles and at most $(1+o(1))\alpha^2n/200$ paths, after some number of iterations we will have a graph containing only one path and at most $(1+o(1))\alpha^2n/100$ cycles.
Furthermore, once there is only one path, there will never be more than one again, so we will never apply case~\hyperlink{case1}{1} again.
From this point on, we apply cases~\hyperlink{case2}{2} or~\hyperlink{case3}{3} intermittently, as needed, and we always either reduce the number of components of the graph, or make sure that we will reduce this number in the next iteration (while not increasing the number of components now).
In particular, this guarantees that every two steps we decrease the number of components.
Thus, at some point we will have a unique component.
If it is not a cycle, an application of case~\hyperlink{case2}{2} will yield the desired cycle.

The alterations performed throughout the process will use some of the `available' edges.
For instance, in some cases we will pick two vertices $x$ and $y$ and an edge $zz'\in Z_0(x,y)$, and use these to alter the graph.
With this alteration, we will remove $zz'$ from $G_{i-1}$ (and our process is correct only if $zz'\in E(G_{i-1})$).
Thus, for future iterations, we need to remove $zz'$ from the set of `available' edges $Z_0(x,y)$.
For all $i\in[t]$ and all pairs of vertices $x,y$, we will denote by $Z_i(x,y)$ a subset of $Z_0(x,y)$ which is `available' for the next iteration.
For all $i\in[t]$, we will always define $G_i$ as $G_i\coloneqq(G_{i-1}\setminus E_1)\cup E_2$, where $E_1$ and $E_2$ are sets of edges of $H\cup G$.
Then, for all $(x,y)\in V^2$, we will set $Z_i(x,y)\coloneqq Z_{i-1}(x,y)\setminus E_1$.
Thus, we will not write this definition explicitly in each case, as it will follow from the definition of $G_i$.

Throughout the process, we will assume that some conditions hold.
In particular, for all $i\in[t]_0$ we will have $\Delta(G_i)\leq2$ and $P\subseteq G_i\subseteq(H\cup G)\setminus U$ (note that these hold for $G_0$ by definition; for larger values of $i$, they will be a consequence of the specific process which we follow) and, for all $(x,y)\in V^2$, that $Z_i(x,y)\subseteq G_i$ and $Z_i(x,y)$ and $P$ are vertex-disjoint (both of which hold by definition).
Note that the conditions that $Z_i(x,y)\subseteq G_i$ and $\Delta(G_i)\leq2$ imply \eqref{equa:degZ} holds for each $Z_i(x,y)$ (which also holds simply by the definition of $Z_i(x,y)$); we will keep referring to \eqref{equa:degZ} when we wish to use this fact.
Furthermore, we will assume that in each step we have 
\begin{equation}\label{equa:Zsize4}
    |Z_{i-1}(x,y)|\geq\alpha^2n/8;
\end{equation}
it will follow from our process that the number of steps $t$ that we perform and the number of edges that become `unavailable' at each step are small, so that \eqref{equa:Zsize4} holds.

We now provide the details of the process.
Let $i\in\mathbb{N}$ and assume we have already defined $G_{i-1}$.
If~$G_{i-1}$ is a cycle of length $n-m$ (which spans $V$), we are done.
Otherwise, we will want to create a new graph $G_i$.
For this, we will need to consider several cases.

\hypertarget{case1}{\textbf{Case 1.}}
Assume at least two of the components of $G_{i-1}$ are paths.
Let $P'$ be one of the path components of $G_{i-1}$, and let its endpoints be $x$ and $y$.
Choose any edge $zz'\in Z_{i-1}(x,y)$ such that $\dist_{G_{i-1}}(zz',\{x,y\})>1$\COMMENT{This condition ensures that $P'\neq xzz'y$, or other possibilities which would make this step not work. This is not a problem since this forbids at most $2$ edges, and we have many more by \eqref{equa:Zsize4}.} and let $G_i\coloneqq(G_{i-1}\setminus\{zz'\})\cup\{xz,yz'\}$.
Depending on the relative position of $zz'$ with respect to $P'$, the total number of components will either decrease by one, remain the same, or increase by one\COMMENT{If $zz'$ lies outside the path, the number of components decreases. If it lies inside $P'$, then it wil either remain the same or increase by one depending on the relative order of $x,z,z',y$ in $P'$.}.

\hypertarget{case2}{\textbf{Case 2.}}
Assume exactly one component of $G_{i-1}$ is a path $P'$.
Let its endpoints be $x$ and $y$.
Consider the following cases.
\begin{enumerate}[label=2.\arabic*.]
    \item\label{case21} Assume that $N_H(x)\cup N_H(y)\nsubseteq U\cup V(P)\cup V(P')$.
    Suppose that there exists some $z\in N_H(x)\setminus(U\cup V(P)\cup V(P'))$ (the other case is analogous).
    Then, choose a vertex $z'\in N_{G_{i-1}}(z)$ (so $zz'$ is an edge of a cycle of $G_{i-1}$) and let $G_i\coloneqq(G_{i-1}\setminus\{zz'\})\cup\{xz\}$.
    This reduces the number of components and increases the length of the unique path $P'$.
    \item\label{case22} Otherwise, we have that $Z_{i-1}(x,y)\subseteq E(P')$.
    Suppose there is an edge $zz'\in Z_{i-1}(x,y)$ such that $\dist_{P'}(x,z')<\dist_{P'}(x,z)$.
    If so, let $G_i\coloneqq(G_{i-1}\setminus\{zz'\})\cup\{xz,yz'\}$.
    In this case, we turn~$P'$ into a cycle.
    \item\label{case23} Otherwise, every $zz'\in Z_{i-1}(x,y)$ satisfies $z,z'\in V(P')$ and $\dist_{P'}(x,z)<\dist_{P'}(x,z')$.
    By \eqref{equa:degZ} and \eqref{equa:Zsize4}, we can find a set $Z_{i-1}'(x,y)\subseteq Z_{i-1}(x,y)$ with 
    \begin{equation}\label{equa:Zsize5}
        |Z_{i-1}'(x,y)|\geq\alpha^2n/16
    \end{equation}
    and such that any two $zz',ww'\in Z_{i-1}'(x,y)$ are vertex-disjoint.
    Now choose two edges $zz',ww'\in Z_{i-1}'(x,y)$ with $\dist_{P'}(x,z)<\dist_{P'}(x,w)$ which minimise $\dist_{P'}(z,w)$ over all possible such pairs of edges.
    We now have two cases.
    \begin{enumerate}[label=2.3.\arabic*.]
        \item If $\dist_{P'}(z',w)=1$, let $G_i\coloneqq(G_{i-1}\setminus\{z'w\})\cup\{xw,yz'\}$.
        This, again, turns $P'$ into a cycle.
        \item Otherwise, let $z'',w''\in V(P')$ be such that $\dist_{P'}(x,z'')=\dist_{P'}(x,z')+1$ and\linebreak{} $\dist_{P'}(x,w'')=\dist_{P'}(x,w)-1$.
        Let $P''$ be the subpath of $P'$ whose endpoints are~$z''$ and~$w''$, and let $\ell\coloneqq\dist_{P'}(z'',w'')$ be the length of $P''$\COMMENT{Note this distance could be $0$.}.
        By an averaging argument using \eqref{equa:Zsize5}, it follows that 
        \[\ell\leq16|V(P')|/(\alpha^2n)=\bigO(\alpha^{-2})=o(\alpha^2n).\]
        By \eqref{equa:Zsize4}, this guarantees that we can pick an edge $z^3w^3\in Z_{i-1}(z'',w'')$ such that $\dist_{G_{i-1}}(z_3w_3,V(P''))\geq1$\COMMENT{No need to subtract $zz'$ or $ww'$ because these are still part of the graph.}\COMMENT{The fact that the chosen edge cannot be incident to any internal vertex of $P''$ follows since $Z\subseteq G_i$ and $\Delta(G_i)\leq2$. But, similarly to case 1, we need to avoid choosing an edge incident to the endpoints of $P''$, as this would actually make our claims false.
        In case 1 we really need the distance to be greater than $1$; here, I believe at least $1$ suffices.}.
        Now, let \[G_i\coloneqq(G_{i-1}\setminus\{z'z'',ww'',z^3w^3\}))\cup\{xw,yz',z''z^3,w''w^3\}.\]
        In this case, if $z^3w^3\in E(P')$, we turn $P'$ into a cycle; otherwise, we combine the vertices of $P'$ and the cycle containing $z^3w^3$ into two new cycles.
    \end{enumerate}
\end{enumerate}

\hypertarget{case3}{\textbf{Case 3.}}
Assume all components of $G_{i-1}$ are cycles.
We consider the following cases.
\begin{enumerate}[label=3.\arabic*.]
    \item\label{case31} For each edge $xy\in E(G_{i-1})$, let $C_{xy}$ be the cycle of $G_{i-1}$ which contains this edge.
    Assume that there exists some $xy\in E(G_{i-1})\setminus E(P)$ satisfying that $Z_{i-1}(x,y)\nsubseteq E(C_{xy})$.
    If so, let $zz'\in Z_{i-1}(x,y)\setminus E(C_{xy})$ and let $G_i\coloneqq(G_{i-1}\setminus\{xy,zz'\})\cup\{xz,yz'\}$.
    This combines two cycles into one.
    \item\label{case32} Otherwise, from \eqref{equa:Zsize4}, it follows that each cycle of $G_{i-1}$ must contain at least $\alpha^2n/8$ vertices.
    Now let $x,y\in V\setminus V(P)$ be such that $x$ and $y$ lie in different cycles (note, in particular, that we may pick a vertex in any of the cycles since \mbox{$|V(P)|\leq4m<\alpha^2n/8$}) and let $zz'\in Z_{i-1}(x,y)$.
    Assume without loss of generality that $zz'$ lies in a cycle other than the one containing $x$. 
    By the definition of $Z_0(x,y)$, this means that $xz$ is an edge of $H$ joining two of the cycles of $G_{i-1}$.
    Now let $w\in N_{G_{i-1}}(x)$
    and let $G_i\coloneqq(G_{i-1}\setminus\{wx,zz'\})\cup\{xz\}$\COMMENT{Here again we could do better in some cases, but it will not change the asymptotic behaviour.}.
    This combines two cycles into a path.
\end{enumerate}

The fact that $\Delta(G_i)\leq2$ follows by the construction in each case, since we only add at most one edge incident to each vertex, and we only do this for vertices which had degree one in $G_{i-1}$ or for which we first delete an incident edge.
The fact that $P\subseteq G_i$ follows since we have made sure not to delete any of the edges of $P$, and $G_i\subseteq(H\cup G)\setminus U$ since we only add edges of $H$.

Now let $C=G_t$ be the graph resulting from the above process.
The argument that we provided when arguing that the process must end shows that $t\leq(1+o(1))\alpha^2n/40$\COMMENT{An easy upper bound: we start with at most $(1+o(1))\alpha^2n/200$ paths.
In the worst case, each of these is turned into two cycles, so at some point (after at most $(1+o(1))\alpha^2n/200$ iterations) we will have at most $(1+o(1))\alpha^2n/100$ cycles.
Now in every two iterations we are guaranteed to decrease the number of components by at least one, so the number of iterations needed to combine all cycles is at most $(1+o(1))\alpha^2n/50$.
This gives a total of at most $(1+o(1))\alpha^2n/40$ iterations}, and this implies that \eqref{equa:Zsize4} holds (indeed, observe that in all cases the number of edges which are deleted from each $Z_i(x,y)$ is at most $3$, so at most $(3+o(1))\alpha^2n/40$ edges are removed, and the conclusion follows from \eqref{equa:Zsize3}).

By the iterative process above, we have proved the existence of a cycle $C$ of length $n-m$ such that $P\subseteq C$.
We must now prove that there is a cycle of length $k$, for all $3\leq k\leq n$. 
Recall that $Z_t(x,y)\subseteq E(C)\setminus E(P)$ for all $(x,y)\in V^2$ and \eqref{equa:Zsize4} holds.
We split our analysis into three cases.

Suppose first that $3\leq k\leq\alpha^2n/20$.
In such a case, consider any subpath $P'\subseteq C$ of length $k-3$, and let its endpoints be $x$ and $y$\COMMENT{When $k=3$ we have a degenerate case with $x=y$.}.
Now choose any edge $zz'\in Z_t(x,y)$ such that $z,z'\notin V(P')$ (the existence of such an edge follows by \eqref{equa:Zsize4} and \eqref{equa:degZ}).
Then, the union of $P'$ and the path $xzz'y$ forms a cycle of length $k$.

Assume next that $n-m\leq k\leq n$.
Consider a set $J\subseteq[m]$ with $|J|=k+m-n$.
In $C$, for each $j\in J$, replace the edge $e_j=z_jz_j'$ by the path $z_ju_jz_j'$.
This yields a cycle of the desired length.

Finally, assume $\alpha^2n/20<k<n-m$.
Consider a subpath $P'\subseteq C$ of length $k-3$ such that $P\subseteq P'$.
Let the endpoints of $P'$ be $x$ and $y$, respectively, and consider the set $Z_t(x,y)$.
We consider the following three cases.
\begin{enumerate}[label=\arabic*.]
    \item Assume that there exists $zz'\in Z_t(x,y)$ such that $z,z'\notin V(P')$.
    Then, the union of $P'$ and the path $xzz'y$ forms a cycle of length $k$.
    \item Otherwise, let $Z_t'(x,y)\coloneqq\{e\in Z_t(x,y):e\subseteq V(P')\}$, so we have $Z_t'(x,y)\subseteq E(P')$ and $|Z_t'(x,y)|\geq|Z_t(x,y)|-2>\alpha^2n/16$ by \eqref{equa:Zsize4}.
    Recall that $Z_t'(x,y)$ and $P$ are vertex-disjoint.
    Suppose there is an edge $zz'\in Z_t'(x,y)$ such that $\dist_{P'}(x,z')<\dist_{P'}(x,z)$.
    If so,\linebreak{} $(P'\setminus\{zz'\})\cup\{xz,yz'\}$ is a cycle of length $k-2$ which contains $P$. 
    To obtain a cycle of length~$k$, replace $e_1=z_1z_1'$ and $e_2=z_2z_2'$ by the paths $z_1u_1z_1'$ and $z_2u_2z_2'$, respectively.
    \item Otherwise, $Z_t'(x,y)\subseteq E(P')$ and all $zz'\in Z_t'(x,y)$ satisfy that $\dist_{P'}(x,z)<\dist_{P'}(x,z')$.
    Recall, again, that $Z_t'(x,y)$ and $P$ are vertex-disjoint.
    We now proceed similarly to case~\ref{case23}
    First, by \eqref{equa:degZ} and the current bound on $|Z_t'(x,y)|$, we may restrict ourselves to a subset of available edges $Z_t''(x,y)\subseteq Z_t'(x,y)$ such that $|Z_t''(x,y)|\geq\alpha^2n/32$ and any two edges $zz',ww'\in Z_t''(x,y)$ are vertex-disjoint.
    Now, choose two edges $zz',ww'\in Z_t''(x,y)$ with $\dist_{P'}(x,z)<\dist_{P'}(x,w)$ which minimise $\dist_{P'}(z,w)$ over all possible pairs.
    Let $P''$ be the subpath of $P'$ whose endpoints are $z'$ and $w$, and let $\ell\coloneqq\dist_{P'}(z,w)$.
    By an averaging argument using the fact that $|Z_t''(x,y)|\geq\alpha^2n/32$, it follows that $\ell\leq32|V(P')|/(\alpha^2n)<m$.
    Then, the graph \mbox{$(P'\setminus E(P''))\cup\{xw,yz'\}$} is a cycle of length $k-\ell$ which contains $P$\COMMENT{This is because $P$ is longer than the distance between the two edges, and since the two edges avoid $P$, it cannot be that we break $P$.}.
    In order to obtain a cycle of length $k$, for each $j\in[\ell]$ replace $e_j=z_jz_j'$ by the path $z_ju_jz_j'$.\qedhere
\end{enumerate}
\end{proof}


\section{Randomly perturbing graphs by a perfect matching}\label{sect:1factor}

Let $\alpha^*\coloneqq\sqrt{2}-1$.
We first show that, for any constant $\alpha<\alpha^*$, there exist $n$-vertex graphs $H$ with $\delta(H)=\alpha n$ such that $H\cup G_{n,1}$ does not a.a.s.~contain a Hamilton cycle\COMMENT{We can also prove the stronger statement that it a.a.s.~does not contain a Hamilton cycle.}.
Indeed, let $H=(A,B,E)$ be a complete unbalanced bipartite graph, where $|A|=\alpha n$ and $|B|=(1-\alpha)n$ (so $d_H(v)=\alpha n$ for all $v\in B$), and let $G$ be a perfect matching on the same vertex set as $H$.
It is easy to check that a necessary condition for $G$ so that $H\cup G$ contains a Hamilton cycle is that $G[B]$ contains at least $(1-2\alpha)n$ edges\COMMENT{Indeed, observe that, since $G$ is a matching, $(H\cup G)[B]$ does not contain any paths of length $2$ (and the same is true for $(H\cup G)[A]$).
Thus, each cycle in $H\cup G$, seen as a word which only shows the set to which each vertex belongs, must be a sequence of vertices which contains at most to consecutive repeated symbols (i.e., think of a sequence of the form $ABABAABABBABBAB\ldots$).
Let $e_A$ and $e_B$ be the number of times that two consecutive vertices of the sequence lie in $A$ or $B$, respectively, and let $v_A$ and $v_B$ be the number of times that there is an `isolated' vertex in the sequence that lies in $A$ or $B$, respectively.
In order to have a Hamilton cycle, we must have the following:
\[
\begin{cases}
2e_B+v_B=|B|=(1-\alpha)n,\\
2e_A+v_A=|A|=\alpha n,\\
e_B+v_B=e_A+v_A.
\end{cases}
\]
Indeed, the first condition is needed so that all vertices in $B$ are used, and the second is required so that all vertices in $A$ are used.
Finally, the third condition is given by the fact that the number of segments of consecutive vertices in a cycle must be the same for $A$ and $B$.
By substituting the third equation into the second, we have that
\[
\begin{cases}
2e_B+v_B=(1-\alpha)n,\\
e_B+v_B+e_A=\alpha n.
\end{cases}
\]
Subtracting the second equation from the first, we conclude that $e_B-e_A=(1-2\alpha)n$, and the claim follows using the fact that $e_A\geq0$.}.
Now, in $G_{n,1}$, each edge appears with probability $1/(n-1)$, so\COMMENT{The last inequality follows since $\alpha<\alpha^*$.}
\[\mathbb{E}[e_{G_{n,1}}(B)]=\binom{(1-\alpha)n}{2}\frac{1}{n-1}\leq\frac{(1-\alpha)^2}{2}n<(1-2\alpha)n.\]
The conclusion then follows by Markov's inequality\COMMENT{Let $\beta\coloneqq(1-2\alpha+(1-\alpha)^2/2)/2$, so this number is strictly between the other two (in particular, strictly less than $1-2\alpha$).
By Markov's inequality, we have that
\[\mathbb{P}[H\cup G_{n,1}\text{ is Hamiltonian}]\leq\mathbb{P}[e(G_{n,1}[B])\leq\beta n]\leq\frac{\mathbb{E}[e(G_{n,1}[B])]}{\beta n}\leq\frac{1-\alpha^2}{2\beta},\]
which is a constant less than $1$.}.

In order to prove \cref{thm:matching}, we first need the following lemma. 

\begin{lemma}\label{lem:partition}
Let $1/n \ll \beta < \alpha/2 < 1/4$.
Let $H$ be an $n$-vertex graph with $\delta(H)\geq\alpha n$ which does not contain a matching of size greater than $(n-\sqrt{n})/2$.
Let $M$ be a maximum matching in $H$.
Then, the vertex set of $H$ can be partitioned into sets $A\cupdot B_1\cupdot B_2\cupdot C_1\cupdot C_2\cupdot R$ in such a way that the following hold:
\begin{enumerate}[label=$(\mathrm{H\arabic*})$]
    \item\label{partP1} $|A| \leq 12\beta^{-2}$;
    \item\label{partP4} $n-2|M|-|A|\leq|R|\leq n-2|M|$;
    \item\label{partP2} $|B_1|=|B_2|$ and $H[B_1,B_2]$ contains a perfect matching; furthermore, for every $v\in B_2\cup R$ we have $e_H(v,B_1)\geq(\alpha -2\beta) n$, and
    \item\label{partP3} $|C_1|=|C_2|$ and $H[C_1,C_2]$ contains a perfect matching; furthermore, for all $v\in C_2$ we have $e_H(v,B_2\cup R)\leq\beta^{-1}+1$.
\end{enumerate} 
\end{lemma}

\begin{proof}
Let $M$ be a maximum matching in $H$.
Let $V\coloneqq V(M)$ and $R'\coloneqq V(H)\setminus V$.
By the maximality of $M$ and the condition on $H$ in the statement, we have that $E(H[R'])=\varnothing$ and $|R'|\geq\sqrt{n}$.
For each vertex $v\in V$, let $M(v)$ denote the unique vertex $w\in V$ such that $vw\in M$.
For any set $S\subseteq V$, we define $M(S)\coloneqq\{M(v):v\in S\}$.

To prove the statement we will follow an iterative process.
We will inductively construct two sequences of sets $\{B_1^{i}\}_{i\geq1}$ and $\{B_2^{i}\}_{i\geq1}$ with $B_1^{i},B_2^{i}\subseteq V$ and use these to construct the vertex partition.
First, for notational purposes we let $B_2^{-1}=B_1^0\coloneqq\varnothing$, and we define $B_2^0\coloneqq R'$.
Then, for each $i\geq1$, we define
\begin{equation}\label{equa:bidef}
    B_1^{i}\coloneqq\left\{v\in V\setminus \left(\bigcup_{j=0}^{i-1} (B_1^{j}\cup B_2^{j})\right) : e_H(v,B_2^{i-1})\geq 2\right\}\quad\text{ and }\quad B_2^{i}\coloneqq  M(B_1^{i}).
\end{equation}

\begin{claim}\label{claim:properties}
For all $i\in[\beta n/2]_0$, the following properties hold.
\begin{enumerate}[label=$(\mathrm{\roman*})$]
    \item\label{item: p3} $B_1^i$ does not span any edge from $M$.
    \item\label{item: p4} $B_2^i$ is an independent set.
    \item\label{item: p5} All but at most $4\beta^{-1}$ vertices $v\in B_2^{i-1}$ satisfy that
    \[e_H\left(v,\bigcup_{j=1}^{i} B_{1}^{j}\right)\geq(\alpha-\beta)n.\]
\end{enumerate}
\end{claim}

\begin{claimproof}
We proceed by induction on $i$.
As a base case, note that \ref{item: p3} holds trivially for $B_1^0$, we have already established that $B_2^0=R'$ is an independent set, and \ref{item: p5} is vacuously true.
Now, for some $i\in[\beta n/2-1]_0$, assume that the properties in the statement hold for all $j\in[i]_0$ and that we want to show the properties also hold for $i+1$.
Note that \eqref{equa:bidef} and property \ref{item: p3} imply that 
\begin{equation}\label{equa:bidisjoint}
\begin{minipage}[c]{0.83\textwidth}\centering
for all $(j,\ell),(j',\ell')\in([i]_0)\times[2]$ with $(j,\ell)\neq(j',\ell')$ we have $B_\ell^j\cap B_{\ell'}^{j'}=\varnothing$.
\end{minipage}\ignorespacesafterend 
\end{equation} 

Observe that, if $|B_1^{i+1}|\leq1$, then \ref{item: p3} and \ref{item: p4} hold trivially, and $B_1^{i+j}=\varnothing$ for all $j\geq2$.
Therefore, for the proof of these two properties we may assume that $|B_1^{i+1}|\geq2$ and, thus, $|B_1^j|\geq2$ for all $j\in[i]$.

In order to prove \ref{item: p3}, suppose for a contradiction that $B_1^{i+1}$ spans some edge $e=xy\in M$.
Our aim is to construct a matching larger than $M$.
To do so, first note that, by the definition of $B_1^{i+1}$ in \eqref{equa:bidef}, we may find two distinct vertices $x_2^i,y_2^i\in B_2^{i}$ such that $xx_2^i,yy_2^i\in E(H)$.
We now proceed recursively for $j\in[i]$, starting with $j=i$ and decreasing its value, as follows:
\begin{itemize}
    \item Let $x_1^j\coloneqq M(x_2^j)$ and $y_1^j\coloneqq M(y_2^j)$, so $x_1^j,y_1^j\in B_1^j$.
    \item Choose two distinct vertices $x_2^{j-1},y_2^{j-1}\in B_2^{j-1}$ such that $x_1^jx_2^{j-1},y_1^jy_2^{j-1}\in E(H)$ (which can be done by the definition of $B_1^j$).
\end{itemize}
Now consider the path $P\coloneqq x_2^0x_1^1x_2^1\ldots x_1^ix_2^ixyy_2^iy_1^i\ldots y_2^1y_1^1y_2^0$ (where \eqref{equa:bidisjoint} guarantees that this is indeed a path).
We may then take $M_1\coloneqq M\triangle E(P)$.
It is easy to verify that $|M_1|=|M|+1$, which contradicts the maximality of $M$. 

In order to prove \ref{item: p4}, we proceed similarly.
Suppose for a contradiction that $B_2^{i+1}$ is not an independent set and let $e=v_2^{i+1}w_2^{i+1}\in E(H[B_2^{i+1}])$.
We now proceed recursively for $j\in[i+1]$, starting with $j=i+1$ and decreasing its value, as follows:
\begin{itemize}
    \item Let $v_1^j\coloneqq M(v_2^j)$ and $w_1^j\coloneqq M(w_2^j)$, so $v_1^j,w_1^j\in B_1^j$.
    \item Choose two distinct vertices $v_2^{j-1},w_2^{j-1}\in B_2^{j-1}$ such that $v_1^jv_2^{j-1},w_1^jw_2^{j-1}\in E(H)$ (which can be done by the definition of $B_1^j$).
\end{itemize}
Consider the path $P'\coloneqq v_2^0v_1^1v_2^1\ldots v_1^iv_2^iv_1^{i+1}v_2^{i+1}w_2^{i+1}w_1^{i+1}w_2^iw_1^i\ldots w_2^1w_1^1w_2^0$ (in order to guarantee that this is indeed a path, we are using the fact that we have now proved that \eqref{equa:bidisjoint} also holds for $i+1$).
We may now take $M_1'\coloneqq M\triangle E(P')$ and, again, verify that $|M_1'|=|M|+1$. 

Finally, suppose there are $k\coloneqq\lceil4\beta^{-1}\rceil$ distinct vertices $v_1,\ldots,v_k\in B_2^i$ such that for all $\ell\in[k]$ we have 
\begin{equation}\label{equa:claimiiineweq}
    e_H\left(v_\ell,\bigcup_{j=1}^{i+1} B_{1}^{j}\right)<(\alpha-\beta)n.
\end{equation}
By \eqref{equa:bidef}, for each $j\in[i-1]_0$ and $\ell\in[k]$ we have that $e_H(v_\ell,B_2^j)\leq1$.
It follows from this, the fact that $E(H[B_2^i])=\varnothing$ (by \ref{item: p4}), the assumption on the minimum degree of $H$ and \eqref{equa:claimiiineweq} that for each $\ell\in[k]$ we have 
\[e_H\left(v_\ell,V\setminus\left(B_{1}^{i+1}\cup\bigcup_{j=1}^{i} (B_{1}^{j}\cup B_2^j)\right)\right)>\beta n-i\geq\beta n/2.\]
Furthermore, again by \eqref{equa:bidef}, we must have $e_H(B_2^i,B_2^{i+1})\leq n$\COMMENT{Indeed, otherwise there would have to be some vertex $v\in B_2^{i+1}$ with $e(v,B_2^i)\geq2$, but this would contradict the definition of $B_1^{i+1}$.}.
Combining this with the previous bound, it follows that
\[e_H\left(B_2^i,V\setminus\bigcup_{j=1}^{i+1} (B_{1}^{j}\cup B_2^j)\right)>k\beta n/2-n\geq n,\]
so there must exist some vertex $y\in V\setminus\bigcup_{j=1}^{i+1} (B_{1}^{j}\cup B_2^j)$ with $e_H(y,B_2^i)\geq2$, which contradicts the definition of $B_1^{i+1}$, so \ref{item: p5} holds.
\end{claimproof}

Consider the smallest $j\geq1$ such that $|B_1^j|<\beta n/2$ and let $i^*\coloneqq j-1$.
Note that $i^*\leq\beta^{-1}<\beta n/2$\COMMENT{We have that $n\geq|\bigcup_{i=1}^{i^*}(B_1^i\cup B_2^i)|\geq 2i^*\beta n/2$, so directly $i^*\leq 1/\beta$.} and that \cref{claim:properties}\ref{item: p5} for $i=1$ guarantees that $i^*\geq1$\COMMENT{For this to be true, we need that $\alpha-\beta\geq\beta/2$; this is guaranteed by the condition that $\beta<\alpha/2$ from the statement (which could be relaxed to this condition, but would yield meaningless results).}.

\begin{claim}\label{claim:properties2}
All but at most $4\beta^{-1}$ vertices $v\in B_2^{i^*}$ satisfy that
\[e_H\left(v,\bigcup_{j=1}^{i^*} B_{1}^{j}\right)\geq(\alpha-\beta)n.\]
\end{claim}

\begin{claimproof}
Similarly to the proof of property \ref{item: p5} in \cref{claim:properties}, suppose there is a set $S\subseteq B_2^{i^*}$ of size $k\coloneqq\lceil4\beta^{-1}\rceil$ such that for every $v\in S$ we have $e_H(v,\bigcup_{j=1}^{i^*} B_1^{j})<(\alpha-\beta)n$.
By \eqref{equa:bidef}, for each $j\in[i^*-1]_0$ and $v\in S$ we have that $e_H(v,B_2^j)\leq1$.
Combining these two bounds with the fact that $E(H[B_2^{i^*}])=\varnothing$ (by \cref{claim:properties}\ref{item: p4}) and the assumption on the minimum degree of $H$, it follows that for each $v\in S$ we have 
\[e_H\left(v,V\setminus\bigcup_{j=1}^{i^*} (B_{1}^{j}\cup B_2^j)\right)>\beta n-i^*,\]
so we conclude that
\[e_H\left(S,V\setminus\bigcup_{j=1}^{i^*} (B_1^{j}\cup B_2^{j})\right)>k(\beta n-i^*).\]
On the other hand, by the definition in \eqref{equa:bidef}, we have that\COMMENT{Each vertex (of which there are at most $n$) which sends at least 2 edges must lie in $B_1^{i^*+1}$ by definition, and all other vertices cannot send more than 1 edge.}
\[e_H\left(S,V\setminus\bigcup_{j=1}^{i^*} (B_1^{j}\cup B_2^{j})\right)\leq n+k|B_1^{i^*+1}|.\]
It follows that $|B_1^{i^*+1}|>\beta n-i^*-n/k\geq\beta n/2$, which contradicts the definition of $i^*$.
\end{claimproof}

Now, a partition of $V(H)$ as described in the statement can be obtained as follows.
Let $W$ be the set of all vertices $v\in\bigcup_{i=1}^{i^*}B_2^{i}$ for which $e_H(v,\bigcup_{i=1}^{i^*} B_{1}^{i})<(\alpha-\beta)n$, and let $W'$ be the set of all vertices $v\in B_2^0$ for which $e_H(v,B_{1}^{1})<(\alpha-\beta)n$.
Let $A\coloneqq W'\cup W\cup M(W)$.
It follows by \cref{claim:properties}\ref{item: p5} together with \cref{claim:properties2} that $|A|\leq 4\beta^{-1}(2i^*+1)\leq 12\beta^{-2}$, so \ref{partP1} holds.
Then, for each $\ell\in[2]$, let $B_\ell\coloneqq\bigcup_{i=1}^{i^*} B_\ell^{i}\setminus A$, and let $R\coloneqq R'\setminus A$ (so \ref{partP4} holds).
It follows that, for every $v\in R\cup B_2$, we have $e_H(v,B_1)\geq(\alpha-\beta)n-|A|\geq(\alpha-2\beta)n$, so \ref{partP2} holds.
Finally, let $C\coloneqq V\setminus(A\cup B_1\cup B_2)$.
Note that, if $C$ is not empty, then it contains a perfect matching by construction.
Consider a bipartition of $C$ into sets $C_1$ and $C_2$ such that $C_2=M(C_1)$ and $B_1^{i^*+1}\subseteq C_1$.
Then, \ref{partP3} follows by \eqref{equa:bidef} and the fact that $i^*\leq\beta^{-1}$.
\end{proof}

We will also use the following observations repeatedly.
\Cref{remark:trivial} is a trivial observation, while \cref{remark:newcycles,remark:newcycles2} follow from elementary case analyses.

\begin{remark}\label{remark:trivial}
Let $G$ be a graph and consider a bipartition $V(G)=A\cupdot B$.
Let $P\subseteq G$ be a path of length at least $5$ which does not contain two consecutive edges in $A$ or in $B$.
Then, for each $X\in\{A,B\}$, the path $P$ contains two distinct vertices $x,y\in X$ with $\dist_P(x,y)\leq3$.\COMMENT{
\begin{proof}
Given any path $P$ of length at least five, consider a string of length $|V(P)|$ made of two symbols, $A$ and $B$, one for each vertex, representing whether it lies in $A$ or in $B$.
The condition in the statements asserts that there never are three consecutive symbols which are equal.
Since the string contains at least $6$ symbols, it must thus contain at least $2$ instances of each of the symbols.
The conclusion that, for each $X\in\{A,B\}$, there exist two instances of $X$ which lie at distance at most $3$ from each other follows immediately by considering consecutive instances of $X$ in the string.
(Assume they are at distance at least $4$. Then, there are at least three consecutive elements in the string distinct from $X$, but this violates the condition in the statement.)
\end{proof}
}
\end{remark}

\begin{remark}\label{remark:newcycles}
Let $G$ be a union of vertex-disjoint paths and cycles, where there are $p$ paths and $c$ cycles.
Let $P_1,P_2\subseteq G$ be two non-degenerate vertex-disjoint subpaths of $G$.
Let $x$ be an endpoint of $P_1$ and~$y$ be an endpoint of $P_2$, and assume $e\coloneqq\{x,y\}\notin E(G)$.
Then, $(G\setminus(P_1\cup P_2))\cup\{e\}$ is a union of vertex-disjoint paths and cycles, with at most $p+1$ non-degenerate paths and at most $c+1$ cycles.
\end{remark}

\begin{remark}\label{remark:newcycles2}
Let $G$ be a union of vertex-disjoint paths and cycles.
Let $x,y,z,z'\in V(G)$ such that $e_1\coloneqq\{x,z\},e_2\coloneqq\{y,z'\}\notin E(G)$ and $x$ and $y$ lie in the same cycle $\cC$ of $G$.
Let $P_1,P_2,P_3\subseteq G$ be three non-degenerate vertex-disjoint subpaths of $G$ such that $P_1$ has $z$ as an endpoint, $P_2$ has $z'$ as an endpoint, and $P_3$ is an $(x,y)$-path.
Then, $(G\setminus(P_1\cup P_2\cup P_3))\cup\{e_1,e_2\}$ may contain a\COMMENT{At most $2$} cycle $\cC'\nsubseteq G$ only if at least one of the following holds: in $G\setminus(P_1\cup P_2\cup P_3)$, $x$ and $z$ lie in the same component, or $y$ and $z'$ lie in the same component, or $z$ and $z'$ lie in the same component.
\end{remark}

We are finally ready to prove \cref{thm:matching}.
The general idea is similar to the proof of \cref{thm:2factor,thm:matching2}, though the details are more involved.
We will use \cref{lem:partition} to show that $H$ satisfies certain properties and combine these with the properties given by \cref{lem:fewcycles,lem:triangsand3paths}.
We will then consider the graph given by the union of $G_{n,1}$ and a large matching of~$H$, given by \cref{lem:partition}, and iteratively modify this graph to obtain cycles of all lengths.
To achieve this, we will construct a special path $P$ which will be used to absorb other vertices, in a similar fashion to the proof of \cref{thm:2factor,thm:matching2}, and then find an almost spanning cycle~$C$ containing~$P$.
A key difference with respect to that proof is that the iterative process by which we construct~$C$ is more involved, since the graph we start with is more structured.
Another key difference is that, while creating $C$, we will isolate some vertices which cannot be absorbed with~$P$; we need to make sure that these can be absorbed by the cycle itself at a later stage, so that we may find a Hamilton cycle (but these vertices play no role when constructing cycles shorter than $C$).
The properties given by \cref{lem:partition} are crucial in proving that our construction works.

\begin{proof}[Proof of \cref{thm:matching}]
By \cref{thm:matching2}, if $H$ contains a matching covering all but $o(n)$ vertices, we are done, so we may assume that the largest matching in $H$ covers at most $n-\sqrt{n}$ vertices.

Let $0<\eta\ll\epsilon\ll1$.
Throughout, we always assume $n$ is sufficiently large.
Apply \cref{lem:partition} with $\alpha\coloneqq(1+\epsilon)(\sqrt{2}-1)$ and $\beta=\eta$.
(Note, in particular, that we may assume $\alpha<1/2$.)
This yields a partition of $V(H)$ into $A\cupdot B_1\cupdot B_2\cupdot C_1\cupdot C_2\cupdot R$ satisfying properties \ref{partP1}--\ref{partP3} of the statement of the lemma.
We define $\gamma_1\coloneqq|C_1|/n$ and $\gamma_2\coloneqq|B_1|/n$, so by \ref{partP1} and \ref{partP4} it follows that 
\begin{equation}\label{equa:Rsize}
    |R|=(1-2\gamma_1-2\gamma_2)n \pm 12\eta^{-2}\geq\sqrt{n}/2
\end{equation}
(where the lower bound follows from the bound on the size of a maximum matching in~$H$).
Note that by \ref{partP2} and since $\alpha>\sqrt{2}-1$ we have\COMMENT{Indeed, we are guaranteed that $|B_1|\geq(\alpha-2\eta)n=(1-2\eta/\alpha)\alpha n\geq(1-6\eta)\alpha n$, where the last inequality follows from the fact that $\alpha>1/3$.}
\begin{equation}\label{extraneweq}
    \gamma_2 \geq (1-6\eta)\alpha.
\end{equation}
Observe that $\gamma_1\leq1/2-\gamma_2$.
It follows from this, \eqref{extraneweq}, the bound on $\alpha$, and by taking $\eta$ sufficiently small that\COMMENT{We have
\begin{align*}
    \gamma_1\leq1/2-\gamma_2\leq1/2-(1-6\eta)\alpha\leq1/2-(1-6\eta)(\sqrt{2}-1)=3/2-\sqrt{2}+6\eta(\sqrt{2}-1)\leq86/1000,
\end{align*}
where for the last inequality we need to take a sufficiently small $\eta$.
}
\begin{equation}\label{equa:gamma1bound}
    \gamma_1\leq86/1000.
\end{equation}
It follows from \eqref{extraneweq}, \ref{partP1}, \ref{partP3}, the fact that $e_H(v,C_1\cup C_2)\leq|C_1\cup C_2|\leq(1-2\gamma_2)n$, and the bounds on $\delta(H)$ and $\alpha$ that\COMMENT{Using that $|C_1\cup C_2|\leq(1-2\gamma_2)n$, we have that
\begin{align*}
    e_H(v,B_1)&=e_H(v,V(H))-e_H(v,B_2\cup R)-e_H(v,C_1\cup C_2)-e_H(v,A)\\
    &\geq\alpha n-\eta^{-1}-1-|C_1\cup C_2|-|A|\geq\alpha n-\eta^{-1}-1-(1-2\gamma_2)n-12\eta^{-2}\\
    &\geq(3\alpha-12\eta-1)n-12\eta^{-2}-\eta^{-1}-1\geq n/5.
\end{align*}}
\begin{equation}\label{equa:CedgestoB}
\begin{minipage}[c]{0.8\textwidth}\centering
for all $v\in C_2$ we have $e_H(v,B_1)\geq n/5$.
\end{minipage}\ignorespacesafterend 
\end{equation}
Indeed, for each $v\in C_2$ we have that
\begin{align*}
    e_H(v,B_1)&=e_H(v,V(H))-e_H(v,B_2\cup R)-e_H(v,C_1\cup C_2)-e_H(v,A)\\
    &\geq\alpha n-\eta^{-1}-1-(1-2\gamma_2)n-12\eta^{-2}\\
    &\geq(3\alpha-12\eta-1)n-12\eta^{-2}-\eta^{-1}-1\geq n/5.
\end{align*}
Furthermore, by \ref{partP2}, for any $x,y\in B_2\cup R$ we have $|N_H(x)\cap N_H(y)\cap B_1|\geq(2\alpha-\gamma_2-4\eta)n$.
It then follows from the fact that $\gamma_2\leq1/2$ and the bound on $\alpha$ that\COMMENT{We have $|N_H(x)\cap N_H(y)\cap B_1|\geq(2\alpha-\gamma_2-4\eta)n=(1-4\eta/(2\alpha-\gamma_2))(2\alpha-\gamma_2)n$.
Now, since $\gamma_2\leq1/2$, we have that $2\alpha-\gamma_2\geq2\sqrt{2}-5/2\geq3/10$, which means that $4/(2\alpha-\gamma_2)\leq40/3<20$.}
\begin{equation}\label{equa:RBintersect}
\begin{minipage}[c]{0.8\textwidth}\centering
for all $x,y\in B_2\cup R$ we have $|N_H(x)\cap N_H(y)\cap B_1|\geq(1-20\eta)(2\alpha-\gamma_2)n$.
\end{minipage}\ignorespacesafterend 
\end{equation}
Let $M$ be the union of a perfect matching in $H[C_1,C_2]$ and a perfect matching in $H[B_1,B_2]$.
For any $v\in V(M)$, we let $M(v)$ be the unique vertex $w$ such that $vw\in M$.
Similarly, for any set $A\subseteq V(M)$, we let $M(A)\coloneqq\bigcup_{v\in A}M(v)$.

Let $G=G_{n,1}$.
Throughout the rest of the proof, we will use the fact that $d_G(x)=1$ for every $x\in V(H)$ repeatedly and without further mention.
We now apply \cref{lem:triangsand3paths}, with $\eta$ playing the role of $\epsilon$, to the pairs of sets $(A_i,B_i)=(N_H(x),N_H(y))$ for all $(x,y)\in V(H)\times V(H)$ (with $\alpha_i=\beta_i=\alpha$) and the pairs of sets $(A_i,B_i)=(N_H(x)\cap N_H(y)\cap B_1,N_H(x)\cap N_H(y)\cap B_1)$ for all pairs $(x,y)\in (B_2\cup R)\times(B_2\cup R)$ (with $\alpha_i=\beta_i=(1-20\eta)(2\alpha-\gamma_2)$, see \eqref{equa:RBintersect}).
Condition on the event that the statement of \cref{lem:triangsand3paths} holds for these pairs of sets, and that the statements of \cref{lem:fewcycles}\ref{item:fewcycles2} and \cref{rem:fewcycles} hold, which occurs a.a.s.
Let $G_0\coloneqq M\cup G$.
We claim that the following properties hold:
\begin{enumerate}[label=(G\arabic*)]
    \item\label{item:PMup2} For every $(x,y)\in V(H)\times V(H)$, we have $e_G(N_H(x), N_H(y))\geq(1-\eta)\alpha^2n/2$.\COMMENT{The division by $2$, both here and in the next one, comes from counting edges twice in the case that $N_H(x)$ and $N_H(y)$ intersect (in the next one it is more clear, as every edge is counted twice).}
    \item\label{item:PMup3} For every $(x,y)\in (B_2\cup R)\times(B_2\cup R)$, we have 
    \[e_G(N_H(x)\cap N_H(y)\cap B_1)\geq (1-50\eta)(2\alpha-\gamma_2)^2n/2.\]
    \item\label{item:PMup4} $G_0$ is the union of at most $\log^2n$ cycles and $(1+o(1))|R|/2$\COMMENT{Here we are using that $|R|=\Omega(\sqrt{n})$, otherwise we would need further terms. Of course, we are also using \ref{partP1}.} paths, all vertex-disjoint.
    \item\label{item:PMup6} Each of the paths of $G_0$\COMMENT{Here, and in many other places, I hope it is clear that we mean `all paths which are connected components', and not all paths that can be found as subgraphs of $G$.} either has both endpoints in $A\cup R$ or is comprised of a single edge in $E(M)\cap E(G)$.
    Furthermore, every vertex in $A\cup R$ is the endpoint of some path.
    \item\label{item:PMup7} All subpaths of $G_0$ alternate between edges of $G$ and edges of $M$.
    \item\label{item:PMup5} $|E(M)\cap E(G)|\leq\log^2n$.
\end{enumerate}
Indeed, both \ref{item:PMup2} and \ref{item:PMup3} follow from \cref{lem:triangsand3paths}, \ref{item:PMup4} follows from \cref{lem:fewcycles}\ref{item:fewcycles2} together with \ref{partP4} and \eqref{equa:Rsize}\COMMENT{The point here is that we need to have $|R|=(1-o(1))(n-2|M|)$, and this follows by \ref{partP4}, and also that $|R|=\omega(\log^2n)$, which follows from the bound on the size of the maximum matching (which we have already incorporated into \eqref{equa:Rsize}).}, and \ref{item:PMup5} holds by \cref{rem:fewcycles}.
Finally, \ref{item:PMup6} and \ref{item:PMup7} must hold by definition.

A simple algebraic manipulation with \eqref{equa:Rsize}, \eqref{extraneweq} and the definition of $\alpha$ shows that
\begin{equation}\label{equa:tightbound}
    (1-50\eta)(2\alpha-\gamma_2)^2n\geq(1+\eta)|R|.
\end{equation}
Using \eqref{extraneweq} it is also easy to check that $(1-3\eta)\alpha^2n\geq(1-50\eta)(2\alpha-\gamma_2)^2n$\COMMENT{We need to check that $(1-3\eta)\alpha^2\geq(1-50\eta)(2\alpha-\gamma_2)^2$.
By \eqref{extraneweq}, it suffices to check that $(1-3\eta)\alpha^2\geq(1-50\eta)(2\alpha-(1-6\eta)\alpha)^2=(1-50\eta)((1+6\eta)\alpha)^2$, that is, it suffices to check that
\[1-3\eta\geq(1-50\eta)(1+6\eta)^2=(1-50\eta)(1+12\eta+36\eta^2)=1-38\eta\pm\bigO(\eta^2),\]
which is clearly true for sufficiently small $\eta$.}, hence
\begin{equation}\label{equa:tightbound2}
    (1-3\eta)\alpha^2 n\geq(1+\eta)|R|.
\end{equation}
The bounds in \eqref{equa:tightbound} and \eqref{equa:tightbound2} are crucial for our proof.

\vspace{1em}

Our goal is to show that we can find cycles of all possible lengths in $H\cup G_0=H\cup G$; to achieve this, we will need to modify $G_0$ through an algorithm, where we will delete some edges of $G_0$ and add some edges of $H$ for each subsequent alteration.
The outcome of this process will be an `almost' spanning cycle $C$, and we will make sure to satisfy certain properties which will allow us to obtain all the desired cycles.

We begin with a high-level sketch of the process.
At each step of the algorithm, we always think of the graph as a union of vertex-disjoint paths and cycles; the vertex set can become smaller at each step, though, as we sometimes `delete' some vertices (which we will need to `absorb' at the end of the process).
We split the process into six steps (see Steps~\hyperlink{Step1}{1}--\hyperlink{Step6}{6} below).
In Step~\hyperlink{Step1}{1}, we simply remove all `isolated' edges of $G_0$, since these will not be useful for us.
In Step~\hyperlink{Step2}{2} we create an absorbing path $P\subseteq H\cup G_0$ that will allow us to incorporate a suitable number of vertices into a cycle; this choice will be used at the end of the process to guarantee that we can construct cycles of all lengths.
Crucially, we must make sure that $P$ is not modified through the remaining steps of the process.
At this point all paths in the graph\COMMENT{Meaning all components which are paths (e.g., $P$ itself does not satisfy this); we hope this distinction is clear throughout.} will have their endpoints in $A\cup B_2\cup R$ (when constructing $P$ we may produce new paths, but we make sure that the new endpoints lie in $B_2$).
We would like to have all their endpoints in $B_2\cup R$ so that we may use \ref{item:PMup3} in the future, so through Step~\hyperlink{Step3}{3} we will make it so that all endpoints are in $B_2\cup R$.
Step~\hyperlink{Step4}{4} is used to guarantee that $P$ does not lie in a cycle; this helps us avoid possible problems in Step~\hyperlink{Step5}{5}, where we turn all cycles into paths (making sure that all the resulting endpoints lie in $B_2\cup R$).
At this point, the graph we are considering is a union of vertex-disjoint paths with all their endpoints in $B_2\cup R$.
Then, using the aforementioned \ref{item:PMup3}, in Step~\hyperlink{Step6}{6} we can iteratively combine all the paths into a single, `almost' spanning path containing $P$, which we later turn into an almost spanning cycle $C$ containing $P$.
The bound given in \eqref{equa:tightbound} is crucial to prove that the almost spanning path can be constructed.
For simplicity of notation, we will always refer to the graph by the same name throughout each of the steps, but the graph is continuously updated in each step of the process.

As already mentioned above, throughout the process we sometimes `delete' some vertices from the graph.
What this means is that they no longer play a role in this process and will not be vertices of the resulting cycle $C$.
We will need to ensure that these vertices can later be `absorbed' into the cycle (some via $P$, and the rest without help from $P$).
We will denote this set of `deleted' vertices by $S$.
We think of these vertices simply as being isolated.
Note, however, that not all vertices of $G_0$ that become isolated through the process are added to $S$, as we will still allow some degenerate paths of length $0$ to be part of the graph.
Thus, we will always explicitly say which vertices are added to $S$ (whenever we do this, we mean that these vertices are removed from (the current version of) $G_0$, so we may keep thinking of $G_0$ as a union of paths and cycles with the desired properties\COMMENT{In particular, isolated vertices in $V(H)\setminus(B_2\cup R)$ would be bad for us, as we can never see them as paths with endpoints in the desired set.}).
In particular, we will always think of the graph at any point throughout the process as a graph on vertex set $V(H)\setminus S$ (and, when choosing vertices for any purpose before the end of Step~\hyperlink{Step6}{6}, we will always avoid $S$, even if not explicitly stated).

When altering the graph throughout the process, we will often need to use some edges of~$G$ which are `available' to us.
To keep track of these, we define a set~$D$ of `unavailable' edges.
In particular, all edges deleted from $G_0$ are automatically added to $D$ (so we will not add them explicitly), but some other edges are added to $D$ to ensure that our process will work (in particular, to `protect' $P$ throughout the process).
Whenever we add edges to $D$ without removing them from the graph, we say so explicitly.

Finally, for some of the alterations, given some vertex $v$, we will want to find a neighbour $b_1\in N_H(v)\cap B_1$.
When using these neighbours, we will want to delete the edge of $M$ containing them (this will help us to guarantee that the resulting graph remains a union of disjoint paths and cycles).
However, we cannot always delete this edge of $M$ (in particular, if it has already been deleted).
Therefore, it will be important for us to keep track of those vertices $b_1\in B_1$ which we cannot choose at any given step (they are also `unavailable').
We will denote the set of these vertices by~$K$.
As happens for $D$, we will often update $K$ implicitly; in particular, whenever an edge $b_1b_2\in E(M)$ with $b_1\in B_1$ is added to $D$, implicitly or explicitly, $b_1$ is added to $K$.
However, we will sometimes explicitly add extra vertices to this set (in particular, to `protect' $P$).
We remark that $K$ and $S$ need not be disjoint.

Let us now describe the steps of our process.
The fact that the sets $S$, $D$ and $K$ are updated immediately throughout the iterations in each of the steps is crucial in some of the choices we make below.
We also remark that we continually use the fact that the graph we consider is a union of vertex-disjoint paths and cycles, often implicitly.

\vspace{1em}

\hypertarget{Step1}{\textbf{Step 1.}}
For each edge $e=xy\in E(M)\cap E(G)$, delete $e$ from $G_0$ and add $x$ and $y$ to $S$.

Let the resulting graph be denoted by $G_1$.
We claim that the following properties hold:
\begin{enumerate}[label=(A\arabic*)]
    \item\label{equa:sizeboundsStep1} $|S|,|D|,|K|\leq2\log^2n$.
    \item\label{item:A1} $G_1$ is the union of at most $\log^2n$ cycles and $(1+o(1))|R|/2$ paths, all vertex-disjoint.
    \item\label{item:A2} Each of the paths of $G_1$ has both endpoints in $A\cup R$.
    Furthermore, $d_{G_1}(x)=1$ for every $x\in A\cup R$.
    \item\label{item:A3} All subpaths of $G_1$ alternate between edges of $G$ and edges of $M$.
\end{enumerate}
Indeed, \ref{equa:sizeboundsStep1} follows immediately by \ref{item:PMup5}, \ref{item:A2} follows by \ref{item:PMup6} and the deletions in Step~\hyperlink{Step1}{1}, and \ref{item:A1} and \ref{item:A3} follow from \ref{item:PMup4} and \ref{item:PMup7}, respectively.

In particular, note that \ref{item:A3} and the fact that all edges of $M$ are contained in $B_1\cup B_2\cup C_1\cup C_2$ implies the following:
\begin{enumerate}[label=(A\arabic*),start=5]
    \item\label{item:A5} If $e\in E(G_1[A\cup R])$, then $e$ is an isolated edge in $G_1$\COMMENT{That is, a component of $G_1$ consisting of this edge only}.
\end{enumerate}


\vspace{1em}

\hypertarget{Step2}{\textbf{Step 2.}}
Let $G_1'\coloneqq G_1$; this is a copy of the `original' $G_1$ which we will not update.
We now wish to construct an absorbing path $P$.
Let $U\subseteq R$ be an arbitrary set of $t\coloneqq\lceil3\eta^{-1}\alpha^{-2}\rceil$ vertices (such a set must exist by \eqref{equa:Rsize} and \ref{equa:sizeboundsStep1}\COMMENT{\ref{equa:sizeboundsStep1} is not really needed here, since by construction we have $R\cap(S\cup K)=\varnothing$, but this makes everything super clear even if you no longer remember what happened in the previous step.}).
Fix a labelling of the vertices in~$U$ as $u_1,\ldots,u_t$.
For each $i\in[t]$,
\begin{enumerate}[label=\arabic*.]
    \item choose an edge $x_iy_i\in E_G(N_H(u_i)\cap B_1)\setminus D$ and add it to $D$ (but do not remove it from $G_1$)\COMMENT{So, because we update $D$ in each step, this guarantees that the edges which are chosen here are distinct. Recall that, since they are distinct, these must indeed be disjoint.}.
\end{enumerate}
These edges $x_iy_i$ will later be part of $P$, and they are added to $D$ in order to `protect' $P$.
Their existence, for each $i\in[t]$, follows from \ref{item:PMup3}, \eqref{equa:tightbound}, \eqref{equa:Rsize}, \ref{equa:sizeboundsStep1}, and the value of $t$.

Let $U_E\coloneqq\{x_i,y_i:i\in[t]\}$ and $U_M\coloneqq M(U_E)\subseteq B_2$. 
Remove all the edges of $M$ incident to $U_E\setminus\{x_1,y_t\}$ from~$G_1$\COMMENT{Recall this means by the definition above that $U_E\setminus\{x_1,y_t\}$ is added to $K$. Since we will later choose edges which do not intersect this set, we are guaranteed that the path which we eventually construct does not contain any of these edges.} (in particular, by \ref{item:A3}, now each edge $x_iy_i$ becomes an isolated edge, and the endpoints of the deleted edges in $U_M$ become endpoints of some paths in the resulting graph) and, in order to `protect' $P$, add $x_1$ and $y_t$ to $K$, even if their respective edges in~$M$ are not removed from the graph.

Next we are going to connect the edges $x_iy_i$ we just chose into a path $P$.
We will achieve this by using a $(y_i,x_{i+1})$-path of length $3$ for each $i\in[t-1]$.
We are going to follow a process to choose a new set of edges in $G$, which will be the central edges of the aforementioned paths.
By the end of the process, we will have constructed a path $P$ satisfying the following properties:
\begin{enumerate}[label=(B\arabic*)]
    \item\label{item:P1} $V(P)\cap U=\varnothing$.
    \item\label{item:P3} $M(x_1),M(y_t)\in B_2$ are the endpoints of $P$.
    \item\label{item:P2} $P$ and $E(G)\setminus D$ are vertex-disjoint.
\end{enumerate}
Together with $P$, we will have a graph $G_2$ which satisfies the following properties:
\begin{enumerate}[label=(B\arabic*),start=4]
    \item\label{item:B0} $P\subseteq G_2$.
    \item\label{item:B1} $G_2$ is the union of at most $2\log^2n$ cycles and $(1+o(1))|R|/2$ paths, all vertex-disjoint.
    \item\label{item:B2} Each of the paths of $G_2$ has both endpoints in $A\cup B_2\cup R$.
    \item\label{item:B3} All subpaths of $G_2\setminus P$ have at most two consecutive vertices in $A\cup R\cup B_2\cup C_2$ or in $B_1\cup C_1$.
    \item\label{item:B101} If $e\in E(G_1'[A\cup R])$, then either $e\in E(P)$ or $e$ is an isolated edge in $G_2$.
\end{enumerate}

Let us now describe the process.
We note that \ref{item:A3} is crucial and will be used implicitly throughout.
Let $W_1\coloneqq U_E\cup M(\{x_1,y_t\})$.
For each $i\in[t]$, we will have a set $W_i$ which will consist of all vertices of $P$ which have already been chosen.
We claim that, at the start of each iteration of the process (that is, for each $i\in[t-1]$), we may assume that the following properties hold:
\begin{enumerate}[label=(B\arabic*),start=9]
    \item\label{equa:sizeboundsStep2} $|S|,|D|,|K|\leq3\log^2n$.
    \item\label{item:B100} $W_1\subseteq W_i$.
    \item\label{item:B14} $|W_i|=2t+2i$.
    \item\label{item:B12} $W_i\cap B_1\subseteq K$.
    \item\label{item:B13} $M(W_i\cap B_2)\subseteq K$.
    \item\label{item:B10} $G_1$ is a union of vertex-disjoint paths and cycles.
    \item\label{item:B11} All of the paths of $G_1$ have both endpoints in $A\cup B_2\cup R$, or have one endpoint in $A\cup B_2\cup R$ and the other is $y_i$, or have one endpoint in $A\cup B_2\cup R$ and the other is $x_t$, or consist of a unique edge $x_jy_j$ for some $j\in[t-1]\setminus[i]$.
    \item\label{item:B17} For any vertex $v\in B_2$ which is an endpoint of some path in $G_1$ we have $M(v)\in K$.
    \item\label{item:B16} All edges in $G_1\setminus G_1'$ are contained in $W_i$ or have one endpoint in $B_1$ and the other in $C_2$.
    \item\label{item:B18} All edges $e\in E(G_1'\setminus G_1)$ with $e\in M$ contained in $C_1\cup C_2$ are incident to $W_i$, or contained in~$S$, or satisfy that $\dist_{G_1'}(e,K)\leq1$.
    \item\label{item:B102} $N_{G_1}(A\cup U)\cap K\subseteq W_i$.
\end{enumerate}
Before starting the process (that is, for $i=1$), \ref{equa:sizeboundsStep2} holds by \ref{equa:sizeboundsStep1}, the value of $t$ and the fact that, so far in this step, we have increased the sizes of $S$, $D$ and $K$ by at most $3t$; \ref{item:B100} is trivial; \ref{item:B14} and \ref{item:B12} follow by the definition of $W_1$ and $K$; \ref{item:B13} holds trivially since $W_1\cap B_2=M(\{x_1,y_t\})\subseteq K$ by definition; \ref{item:B10} holds by \ref{item:A1}, since so far in this step we have only deleted edges; \ref{item:B11} and \ref{item:B17} follow from \ref{item:A2} and the fact that, so far, throughout this step we have only deleted edges of $M$ with one endpoint in $B_2$ (which then becomes an endpoint of some path) and the other endpoint in one of the edges $x_iy_i$ (but not the edges containing $x_1$ or $y_t$); \ref{item:B16} is vacuously true since $E(G_1\setminus G_1')=\varnothing$; \ref{item:B18} is vacuously true since all edges of $G_1'$ which have been deleted so far are contained in $B_1\cup B_2$, and \ref{item:B102} holds by construction, since at this point we have $K\subseteq W_1\cup S$ and $N_{G_1}(A\cup U)\cap S=\varnothing$ by the definition of $S$\COMMENT{That is, the idea that $G_1$ has vertex set $V(H)\setminus S$ by definition.}.
We will later show that these properties must indeed hold throughout.

For each $i\in[t-1]$, we proceed as follows:
\begin{enumerate}[label=\arabic*.,start=2]
    \item \hypertarget{step2.1}{Choose} an edge $w_iz_i\in E_G(N_H(y_i)\setminus U,N_H(x_{i+1})\setminus U)\setminus D$\COMMENT{Note in particular that, by removing $D$, we cannot choose an edge in $E(M)$, because these have already been removed.} such that
    \begin{enumerate}[label=(B\arabic*),start=20]
        \item\label{B21} $\dist_{G_1}(\{w_i,z_i\},W_i)\geq5$,
        \item\label{B22} $\dist_{G_1'}(\{w_i,z_i\},W_i)\geq2$, and
        \item\label{B23} $\dist_{G_1'}(\{w_i,z_i\},K)\geq5$.
    \end{enumerate}
    Note that such an edge must exist by \ref{item:PMup2}, \ref{equa:sizeboundsStep2}, \ref{item:B14}, \ref{item:B10}, and the value of $t$\COMMENT{From the lower bound on the number of edges given by \ref{item:PMup2}, we have to subtract at most $2|U|$ edges, then $|D|$ edges, and then all edges `close' to any vertex in $W_i$ (where we want to use the fact that $\Delta(G_0)=2$ to guarantee that these are just constant many).
    Furthermore, recall that we must also subtract all edges incident to $S$, which are at most $2|S|$.}.
    In order to `protect' $P$, add $w_iz_i$ to $D$ (but do not remove it from $G_1$).
    \item Now we make some modifications to ensure $G_1$ will remain a union of vertex-disjoint paths and cycles after adding $y_iw_i$ and $z_ix_{i+1}$.
    If the component of $G_1$ containing $w_iz_i$ is a cycle of length at most $8$, we remove all edges of this cycle (except $w_iz_i$) from $G_1$ and add all its vertices (except $w_i$ and $z_i$) to $S$\COMMENT{This is needed because, for very short cycles, case 4 below would read things which are not consistent. I have verified that cycles of length $4$ and $6$ are a problem, I think cycles of length $8$ or higher are okay. But since, because of the other cases, we remove up to $8$ edges, here it does not bother us to consider up to $8$ edges to remove, and then it is easier to check that the possibilities for $x$ do not interfere with each other.}.
    Otherwise, for each $x\in\{w_i,z_i\}$, we consider several cases:
    \begin{enumerate}[label=3.\arabic*.]
        \item \hypertarget{step2.2.1}{If} $x\in A\cup R$, do nothing.
        \item \hypertarget{step2.2}{If} $x\in B_1$, remove the edge of $M$ containing $x$ from $G_1$ (which must have belonged to $G_1$ by the definition of $K$ and \ref{B23}\COMMENT{By \ref{item:B11} and the choice of $x$ we know that $x$ cannot be an endpoint of any path; but this seems unnecessary since, by \ref{B23}, we are already guaranteed that the edge is there.})\COMMENT{Recall again that this implies we add $x$ to $K$.}.
        Note that $M(x)\in B_2$ becomes an endpoint of a path in the resulting graph.\COMMENT{Note: here and throughout, we sometimes abuse language when we claim that it becomes an endpoint. Indeed, it could be that this vertex was already an endpoint of some path, and in such a case, what we have is that the vertex becomes a degenerate path. In some sense, this vertex would be an endpoint with multiplicity $2$. (If one is careful, one can prove this actually does not happen here, but it will happen in other cases, and also we do not want to be so careful.)}
        \item \hypertarget{step2.3}{If} $x\in C_1$, consider the edge $xy\in E(M)$ (which must lie in $G_1$ by \ref{item:B18}, \ref{B21}, \ref{B22} and \ref{B23}\COMMENT{By \ref{item:B18}, if $xy\notin E(G_1)$, we must have that $x\in W_i$ (which contradicts \ref{B21}) or $y\in W_i$ (which contradicts \ref{B22}) or $\dist_{G_1'}(xy,K)\leq1$ (which contradicts \ref{B23}) or $xy\subseteq S$, which makes no sense.}) and let $z^*$ be the other neighbour of $y\in C_2$ in~$G_1$ (note that it must exist by \ref{item:B10} and \ref{item:B11}\COMMENT{Observe that this suffices since we know that $W_1\subseteq B_1\cup B_2$, so we do not need to worry about the other endpoints which are possible in \ref{item:B11}.}).
        Choose some vertex $z\in (N_H(y)\cap B_1)\setminus (K\cup N_{G_1}(A\cup U)\cup \{w_i,z_i,z^*\})$, which must exist by \eqref{equa:CedgestoB}, \ref{equa:sizeboundsStep2}, \ref{partP1}, and the value of $t$.
        Then, add $yz$ to $G_1$ and remove both $xy$ and the edge of $M$ containing $z$ (which must have belonged to $G_1$ by the definition of $K$).
        Now $M(z)\in B_2$ becomes an endpoint of a path in the resulting graph.
        \item \hypertarget{step2.4}{If} $x\in B_2\cup C_2$, consider the edge $xy\in E(M)$ (which must lie in $G_1$ by \ref{item:B18}, \ref{B21}, \ref{B22} and \ref{B23}\COMMENT{Assume first that $x\in B_2$.
        If $xy\notin E(G_1)$, then we must have $y\in K$.
        But then $x$ would not have been chosen by \ref{B23}.
        Assume now that $x\in C_2$.
        If $xy\notin E(G_1)$, then by \ref{item:B18} we must have $y\in W_i$ (the other options are $x\in W_i$, which would contradict \ref{B21}, or $\dist_{G_1'}(xy,K)\leq1$, which would contradict \ref{B23}, or $xy\subseteq S$, which just makes no sense).
        But this contradicts \ref{B22}.}) and let $z$ be the other neighbour of $y$ in~$G_1$ (which must exist by \ref{item:B100}, \ref{item:B10}, \ref{item:B11} and \ref{B21}\COMMENT{By \ref{item:B10} we must have that $y$ has degree one or two, and we want to prove it has degree $2$.
        Assume it has degree one.
        Then, by \ref{item:B11}, it must be one of the vertices $x_i$ or $y_i$ (since all other endpoints must not lie in $B_1$ and this is the only place where $y$ could lie and be an endpoint). 
        But by \ref{item:B100} and \ref{B21}, in this case we would not have chosen $x$.}).
        Now consider the following cases:
        \begin{enumerate}[label=3.4.\arabic*.]
            \item \hypertarget{step2.4.0}{If} $z\in A\cup R$, remove $xy$ and $yz$ from $G_1$ and add $y$ to $S$.
            We now think of $z$ as a degenerate path.
            \item \hypertarget{step2.4.1}{If} $z\in B_2$, let $z'$ be the other neighbour of $z$ (which must exist by \ref{item:B10}, \ref{item:B17}, \ref{item:B16} and \ref{B23}\COMMENT{By \ref{item:B10}, we know that $z$ must have degree one or two. We want to prove that it must have degree two.
            Assume it has degree one.
            Then, it must be the endpoint of some path.
            Then, by \ref{item:B17}, we must have that $M(z)\in K$. We now want to claim that this contradicts \ref{B23}.
            However, this is not immediately obvious, since it could be that the path $xyzM(z)$ contains some edge not in $G_1'$.
            But we already know that $xy\in E(M)$ and, by \ref{item:B16}, all added edges have one endpoint in $B_1$ and the other in $C_2$. Since $z\in B_2$, the path cannot contain any of these newly added edges, so it must have been a path of length $3$ in $G_1'$. So we reach our desired contradiction.}) and observe that, by \ref{item:B16} and \ref{B23}, we must have $zz'\in M$\COMMENT{Again, if this was not the case, we would have $M(z)\in K$, but then the path $xyzM(z)$ would contradict \ref{B23} (where in order to prove the contradiction we must use \ref{item:B16}).}.
            Remove $xy$ and $yz$ from $G_1$ and add $y$ to $S$.
            Furthermore, add $z'$ to $K$ (but do not remove $zz'$ from $G_1$).
            Then, $z\in B_2$ becomes the endpoint of a path in the resulting graph.
            \item If $z\in B_1$, let $zz'\in E(M)$ (which must lie in $G_1$ by \ref{item:B16}, \ref{B23} and the definition of~$K$\COMMENT{Assume otherwise.
            Then, by definition of $K$, we must have $z\in K$.
            But then, the path $xyz$ of length $2$ contradicts \ref{B23} (this contradiction uses the fact that no edge of this path has been added throughout this step, as follows from \ref{item:B16}).}), remove $xy$, $yz$ and $zz'$ from $G_1$, and add $y$ and $z$ to $S$.
            Now $z'\in B_2$ becomes an endpoint of a path.
            \item \hypertarget{step2.4.3}{If} $z\in C_2$, let $z^*$ be the other neighbour of $z$ in $G_1$ (which must exist by \ref{item:B10} and \ref{item:B11}), and choose some vertex $z'\in (N_H(z)\cap B_1)\setminus(K\cup N_{G_1}(A\cup U)\cup\{w_i,z_i,z^*\})$\COMMENT{At first thought, it must be that $z^*\in C_1$ and thus we do not need to remove it. However, we are applying changes throughout the process that might make $z^*$ change its position (in particular, $z^*\in B_1$ seems possible because of \ref{item:B16}). So to make our lives easier, we just forbid it.} (which must exist by \eqref{equa:CedgestoB}, \ref{equa:sizeboundsStep2}, \ref{partP1}, and the value of $t$).
            Then, add $zz'$ to $G_1$ and remove $xy$, $yz$ and the edge of $M$ containing $z'$ (which must lie in $G_1$ by the definition of $K$) from $G_1$, and add $y$ to $S$.
            Now $M(z')\in B_2$ becomes an endpoint of a path.
            \item \hypertarget{step2.4.4}{If} $z\in C_1$, let $zz'\in E(M)$ (which must lie in $G_1$ by \ref{item:B16}, \ref{item:B18}, \ref{B21} and \ref{B23}\COMMENT{Assume it does not.
            Then, by \ref{item:B18}, we must have that $z\in W_i$ (which, by \ref{item:B16}, would contradict \ref{B21}, since the path $xyz$ would be contained in $G_1'$) or $z'\in W_i$ (which, also by \ref{item:B16}, would contradict \ref{B21}) or $\dist_{G_1'}(zz',K)\leq1$ (which would contradict \ref{B23}) or $zz'\subseteq S$ (which does not make sense).}), let $z^*$ be the other neighbour of $z'\in C_2$ in $G_1$ (which must exist by \ref{item:B10} and \ref{item:B11}), and choose some vertex $z''\in (N_H(z')\cap B_1)\setminus(K\cup N_{G_1}(A\cup U)\cup\{w_i,z_i,z^*\})$ (which must exist by \eqref{equa:CedgestoB}, \ref{equa:sizeboundsStep2}, \ref{partP1}, and the value of $t$).
            Then, add $z'z''$ to $G_1$ and remove $xy$, $yz$, $zz'$ and the edge of $M$ containing $z''$ (which must lie in $G_1$ by the definition of $K$) from $G_1$, and add $y$ and $z$ to $S$.
            The vertex $M(z'')\in B_2$ becomes an endpoint of a path.
        \end{enumerate}
    \end{enumerate}
    \item Add $y_iw_i$ and $z_ix_{i+1}$ to $G_1$.
    Let $W_{i+1}\coloneqq W_i\cup\{w_i,z_i\}$ and iterate.
\end{enumerate}

Note now that \ref{equa:sizeboundsStep2} follows from \ref{equa:sizeboundsStep1} and the fact that $S$, $K$ and $D$ increase their sizes by at most~$8$ in each step of the process above\COMMENT{Each of these can potentially be bounded by a smaller number, but we do not really care what this number is. The upper bound is given by case \hyperlink{step2.4.4}{3.4.5}.}.
This, combined with the at most $3t$ increase before the process and the value of $t$, immediately yields the bound.
\ref{item:B100} holds trivially by definition.
\ref{item:B14} holds trivially by the definition of $W_{i+1}$, since $w_i,z_i\notin W_i$ by \ref{B21} and are distinct.
\ref{item:B12} and \ref{item:B13} hold because, throughout the process, if $x\in B_1\cup B_2$, we delete the edge of $M$ containing $x$, which implies the conditions.
\ref{item:B10} follows from the fact that, throughout, we guarantee that all vertices of the graph~$G_1$ have degree at most $2$: indeed, we make sure to delete one edge incident to each of the vertices that lie in one of the edges which will be added; the only exception to this is case~\hyperlink{step2.2.1}{3.1}, in which there is no need to delete edges by \ref{item:A2} and the fact that the different vertices $x$ throughout the process are always distinct.
In order to prove \ref{item:B11}, note that all the newly created endpoints lie in $B_2$, as remarked throughout the process.
Furthermore, the edges which are added to the graph in cases~\hyperlink{step2.3}{3.3}, \hyperlink{step2.4.3}{3.4.4} and~\hyperlink{step2.4.4}{3.4.5} have one of their endpoints in $C_2$ and the other in $B_1\setminus K$, which by \ref{item:B12} cannot belong to any of the vertices in $W_i$; this guarantees that the different paths consisting of a single edge $x_jy_j$ with $j\in[t-1]\setminus\{1\}$ do not become part of any longer paths or cycles until the iteration in which $w_{j-1}z_{j-1}$ is considered.
Observe that avoiding the vertex $z^*$ in cases~\hyperlink{step2.3}{3.3}, \hyperlink{step2.4.3}{3.4.4} and~\hyperlink{step2.4.4}{3.4.5} is crucial, as otherwise we would create an endpoint in $C_2$.
\ref{item:B17} follows directly from the process, since every time we create a new endpoint $v\in B_2$ we do so by deleting the edge of~$M$ containing $v$, except in case~\hyperlink{step2.4.1}{3.4.2}, where we artificially add $M(v)$ to $K$.
\ref{item:B16} holds directly by construction: in the $i$-th iteration, we add two edges, $y_iw_i$ and $z_ix_{i+1}$, which are contained in $W_{i+1}$ by definition, and possibly some edges with an endpoint in $B_1$ and the other in $C_2$ (this happens in cases~\hyperlink{step2.3}{3.3}, \hyperlink{step2.4.3}{3.4.4} and~\hyperlink{step2.4.4}{3.4.5}).
Now, consider \ref{item:B18}.
In case~\hyperlink{step2.3}{3.3}, the deleted edge is incident to a vertex which is added to $W_{i+1}$; the same is true for the edge containing $x$ in all subcases of case~\hyperlink{step2.4}{3.4}\COMMENT{This is relevant in the case when $x\in C_2$.}; by \ref{item:B16}, we are guaranteed that the other edge $e$ deleted in case~\hyperlink{step2.4.4}{3.4.5} now satisfies $\dist_{G_1'}(e,K)=1$\COMMENT{The vertex $z$ in that case is incident to $y$, and by \ref{item:B16} we are guaranteed that this edge must have belonged to $G_1'$; the vertex $y$ is then added to $K$ (in fact, to $S$), so we have what we wanted.}; the only other case in which one such edge may be deleted is when $w_iz_i$ is contained in a cycle $C$ of length at most $8$, but here all deleted edges are either incident to $W_{i+1}$ or contained in $S$.
Finally, \ref{item:B102} can be checked throughout the process by \ref{item:B11}, \ref{item:B16} and the choices in cases~\hyperlink{step2.3}{3.3}, \hyperlink{step2.4.3}{3.4.4} and~\hyperlink{step2.4.4}{3.4.5}\COMMENT{Choose some $x\in N_{G_2}(A\cup U)\cap K$. By \ref{item:B11}, this vertex cannot be an endpoint, but since it lies in $K$ and $N_{G_1}(A\cup U)\cap S=\varnothing$, either it lies in the path (in which case, we are done) or it must have seen its edge deleted, so since it is not an endpoint, it must have had some edge added to it. But by \ref{item:B16}, all added edges are contained in $V(P)$ (so done) or have an endpoint in $B_1$ and the other in $C_2$, as chosen in cases~\hyperlink{step2.3}{3.3}, \hyperlink{step2.4.3}{3.4.4}. But in all these cases we avoided $N_{G_2}(A\cup U)\cap K$ by design.}.

Note also that \ref{equa:sizeboundsStep2}--\ref{item:B13} hold after the last iteration of the process (that is, for $i=t$).
Similarly, \ref{item:B16} and \ref{item:B102} also hold after the process (that is, replacing $G_1$ by $G_2$ and with $i=t$).
These will come in useful later.

Let $G_2$ be the graph resulting from the process above.
Let 
\[P\coloneqq M(x_1)x_1y_1w_1z_1x_2\ldots y_{t-1}w_{t-1}z_{t-1}x_ty_tM(y_t)\]
(recall that all these vertices are distinct, as follows from the choices throughout the process).
Observe that \ref{item:P1}, \ref{item:P3} and \ref{item:P2} hold by definition.
In order to prove that \ref{item:B0} holds, note first that all the edges of $P$ must belong to $G_1$ at some point throughout the process, since $E(P)\cap E(G)\subseteq E(G_1')$, $x_1M(x_1),y_tM(y_t)\in E(G_1')$ and all other edges of $P$ are added in the fourth step of the process throughout the different iterations.
The fact that none of these edges are deleted throughout the process follows from \ref{item:B12}, \ref{item:B13} and \ref{B21}: indeed, throughout the process, all deleted edges either lie at distance at most $3$\COMMENT{With $3$ being achieved only in the case of a cycle of length $8$ containing $w_iz_i$.} from $w_i$ or $z_i$, so they cannot be incident to $W_i$ by \ref{B21}, or they have an endpoint in $B_1\setminus K$ and belong to $M$ (where we further ensure that said endpoint cannot be either~$w_i$ or~$z_i$), and therefore cannot be incident to $W_i$ by \ref{item:B12} and \ref{item:B13}.
Note, furthermore, that the edges deleted in cases~\hyperlink{step2.2}{3.2}, \hyperlink{step2.3}{3.3} and~\hyperlink{step2.4}{3.4} for $w_i$ cannot `interfere' with those deleted for $z_i$, since in these cases we are guaranteed that $\dist_{G_i\setminus\{w_iz_i\}}(w_i,z_i)\geq9$.
\ref{item:B1} holds by \ref{item:A1}, \ref{item:B10} and since, by \cref{remark:newcycles}, the number of paths and cycles does not increase too much throughout the process.
Indeed, the number of paths increases by at most~$2t$ due to the deletions before the process, and then by at most~$2t$ due to the deletions throughout the process, and at most~$2t$ new cycles are created overall (see \cref{remark:newcycles} for cases~\hyperlink{step2.3}{3.3}, \hyperlink{step2.4.3}{3.4.4} and~\hyperlink{step2.4.4}{3.4.5}).
This, combined with the bounds in \ref{item:A1}, \eqref{equa:Rsize} and the value of $t$, guarantees that \ref{item:B1} holds.
\ref{item:B2} is a direct consequence of \ref{item:B11} together with the final iteration of the process.
Now, by \ref{item:A3} we know that all subpaths of $G_1'$ have at most two consecutive vertices in $A\cup R\cup B_2\cup C_2$ or in $B_1\cup C_1$. 
By \ref{item:B16}, all edges added throughout the process are either contained in $W_t=V(P)$\COMMENT{In particular, contained in $P$ by definition.} or have one endpoint in $B_1$ and the other in $C_2$.
This guarantees that all subpaths of $G_2$ (except possibly those that intersect $P$) must also have at most two consecutive vertices in $A\cup R\cup B_2\cup C_2$ or in $B_1\cup C_1$.
That is, \ref{item:B3} holds.
Finally, \ref{item:B101} follows directly from \ref{item:A5} and \ref{item:B16}.

\vspace{1em}

\hypertarget{Step3}{\textbf{Step 3.}}
Our next goal is to obtain a graph $G_3$ which satisfies the following properties:
\begin{enumerate}[label=(C\arabic*)]
    \item\label{item:C1} $P\subseteq G_3$.
    \item\label{item:C2} $G_3$ is the union of at most $3\log^2n$ cycles and $(1+o(1))|R|/2$ paths, all vertex-disjoint.
    \item\label{item:C3} Each of the paths of $G_3$ has both endpoints in $B_2\cup R$.
    \item\label{item:C4} All subpaths of $G_3\setminus P$ have at most two consecutive vertices in $R\cup B_2\cup C_2$ or in $B_1\cup C_1$.
    \item\label{item:C5} $U\cap V(G_3)=\varnothing$.
\end{enumerate}
The goal of \ref{item:C5} is to have all vertices in $U$ isolated, so that they can later be absorbed by $P$.

In order to achieve this, roughly speaking, we are going to remove all edges incident to $(A\cup U)\setminus V(P)$ from $G_2$ and then perform a few more alterations on the graph to guarantee that \ref{item:C3} holds.
For these alterations, we are are going to follow a process.
We claim that, throughout, we may assume that the following properties hold:
\begin{enumerate}[label=(C\arabic*),start=6]
    \item\label{equa:sizeboundsStep3} $|S|,|D|,|K|\leq4\log^2n$.
    \item\label{item:C7} $G_2$ is a union of vertex-disjoint paths and cycles.
    \item\label{item:C8} Each of the paths of $G_2$ has both endpoints in $A\cup B_2\cup R$.
    \item\label{item:C9} All edges in $G_2\setminus G_1'$ are contained in $V(P)$ or have one endpoint in $B_1$ and the other in $C_2$.
\end{enumerate}
Note that, before we start the process, \ref{equa:sizeboundsStep3}, \ref{item:C7}, \ref{item:C8} and \ref{item:C9} are guaranteed by \ref{equa:sizeboundsStep2}, \ref{item:B1}, \ref{item:B2} and \ref{item:B16}, respectively.

We proceed as follows.
For each edge $e=xy\in E(G_2)$ with $x\in (A\cup U)\setminus V(P)$, we proceed as follows depending on the position of $y$.
Note that, by \ref{item:B1}, we are guaranteed that $y\notin V(P)\setminus M(\{x_1,y_t\})$\COMMENT{Assume $y\in V(P)\setminus M(\{x_1,y_t\})$. By \ref{item:B1}, we are guaranteed that $y$ cannot have degree $3$, so we must have $x\in V(P)$, a contradiction.}.
\begin{enumerate}[label=\arabic*.]
    \item If $y\in A\cup U\cup B_2$, remove $xy$ from $G_2$.
    \item If $y\in R\setminus U$, remove $xy$ from $G_2$ and add $y$ to $S$.
    Note that, by \ref{item:B101}, this does not delete any other edges of $G_2$.
    \item If $y\in B_1$, remove $xy$ and the edge of $M$ containing $y$ from $G_2$ (note that this edge must have been in $G_2$ by \ref{item:B102} and the definition of $K$, since $y\notin V(P)$; see \ref{item:P3}) and add $y$ to $S$. 
    Observe that $M(y)\in B_2$ becomes and endpoint in the resulting graph.
    \item \hypertarget{step3.3}{If} $y\in C_2$, let $z^*$ be the other neighbour of $y$ in $G_2$ (which must exist by \ref{item:C7} and \ref{item:C8}) and choose a vertex $z\in(N_H(y)\cap B_1)\setminus(K\cup\{z^*\})$ (its existence follows from \eqref{equa:CedgestoB} and \ref{equa:sizeboundsStep3}), add $yz$ to $G_2$, and remove $xy$ and the edge of $M$ containing $z$ from $G_2$. 
    \ref{item:C8} guarantees that $z$ does not become an endpoint in the resulting graph. 
    Instead, $M(z)\in B_2$ becomes an endpoint.
    \item \hypertarget{step3.4}{If} $y\in C_1$, let $z$ be the other neighbour of $y$ in $G_2$ (which exists by \ref{item:C7} and \ref{item:C8}), and note that $yz\in E(M)$ (which follows from \ref{item:C9}\COMMENT{By \ref{item:C7}, we know that $y$ must have degree $2$. If $yz\notin M$, we must have added some other edge incident to $y$, but since $y\in C_1$, \ref{item:C9} guarantees this is not the case.}), so $z\in C_2$.
    Let $z^*$ be the other neighbour of $z$ in $G_2$ (which must exist, again, by \ref{item:C7} and \ref{item:C8}).
    Then, choose a vertex $z'\in(N_H(z)\cap B_1)\setminus(K\cup\{z^*\})$ (recall that its existence is guaranteed by \eqref{equa:CedgestoB} and \ref{equa:sizeboundsStep3}), add $zz'$ to $G_2$, remove $xy$, $yz$ and the edge of $M$ containing $z'$ from $G_2$, and add $y$ to $S$.
    Again, by \ref{item:C8}, $z'$ does not become an endpoint, and $M(z')\in B_2$ does.
\end{enumerate}
Once this is done for all the desired edges, add $A\cup U$ to $S$.

Note that, by construction, the sizes of $S$, $D$ and $K$ increase by at most $3$ for each edge $xy$ deleted following the previous process.
It then follows by \ref{partP1}, the definition of $U$ in Step~\hyperlink{Step2}{2} and \ref{equa:sizeboundsStep2} that \ref{equa:sizeboundsStep3} holds throughout, and even after the last step of the process\COMMENT{They were all bounded by $3\log^2n$, and now increase by at most $3\cdot 2(|A|+|U|)\ll\log^2n$ (note the multiplication by $2$ is needed since some vertices in $A$ and $U$ may have degree $2$ in the current graph).}.
The remaining properties follow similarly as in Step~\hyperlink{Step2}{2}.
\ref{item:C7} follows from the fact that, throughout, we guarantee that all vertices of $G_2$ have degree at most $2$, since we delete one edge incident to each of the vertices that lie in one of the edges which will be added.
\ref{item:C8} holds since all the newly created endpoints lie in $B_2$, as remarked throughout the process.
Observe that avoiding $z^*$ in cases~\hyperlink{step3.3}{4} and \hyperlink{step3.4}{5} ensures that we do not create a new endpoint in $C_2$.
Finally, \ref{item:C9} holds since the edges added to the graph in cases~\hyperlink{step3.3}{4} and \hyperlink{step3.4}{5} have one endpoint in $B_1$ and the other in $C_2$ by construction.
Note that \ref{item:C9} also holds after the process, that is, replacing $G_2$ by $G_3$.

Let $G_3$ be the graph resulting from the process.
\ref{item:C1} must hold by \ref{item:B0} and since no edges incident to the internal vertices of $P$ are added nor deleted throughout the process\COMMENT{Edges incident to the endpoints, which lie in $B_2$ by \ref{item:P3}, may have been removed, but this is not a problem since the endpoints are in $B_2$.} (which follows from \ref{item:B12}, \ref{item:B13} and the fact that $y\notin V(P)\setminus M(\{x_1,y_t\})$ throughout).
Furthermore, the process above may again increase the number of paths and cycles, but, by \cref{remark:newcycles}, it creates at most one new cycle or path for each vertex $y$ above.
This, combined with \ref{partP1}, the definition of $U$ in Step~\hyperlink{Step2}{2}, \eqref{equa:Rsize} and \ref{item:B1}, ensures that \ref{item:C2} holds.
\ref{item:C3} follows by \ref{item:C8} since, at the end of the process, we have $A\cap V(G_3)=\varnothing$; similarly, we have that \ref{item:C5} holds too.
Finally, as happened in Step~\hyperlink{Step2}{2}, all edges added in cases~\hyperlink{step3.3}{4} and~\hyperlink{step3.4}{5} above have one endpoint in $B_1$ and the other in $C_2$ (see \ref{item:C9}), so by \ref{item:B3} we must have that \ref{item:C4} holds.

\vspace{1em}

\hypertarget{Step4}{\textbf{Step 4.}}
We now want to make sure that $P$ does not lie in a cycle.
If $P$ is not contained in a cycle, we are done, so assume it is contained in some cycle $\cC$.
If $|E(\cC)|\leq|E(P)|+14$, simply delete $E(\cC)\setminus E(P)$, and add $V(\cC)\setminus V(P)$ to~$S$.
Otherwise, for each $x\in \{M(x_1),M(y_t)\}$ (recall these are the endpoints of $P$ and that they lie in $B_2$, see \ref{item:P3}), we proceed as follows.
Let $y$ be the first vertex of $\cC$ which lies in $B_2\cup C_2$ when moving along $\cC$ away from $P$ starting at $x$ (disregarding $x$ itself).
Note that, by \ref{item:C4} and \cref{remark:trivial}, we have that $\dist_{\cC}(x,y)\leq3$.
Now, delete all edges of the shortest \mbox{$(x,y)$-path}\COMMENT{In particular, $D$ is updated here!} of $\cC$, and add all its internal vertices to $S$.
If $y\in B_2$, it becomes an endpoint and we are done; otherwise, choose a vertex $z\in(N_H(y)\cap B_1)\setminus(K\cup N_{G_3}(y))$ (which must exist by  \eqref{equa:CedgestoB} and \ref{equa:sizeboundsStep3}), add $yz$ to $G_1$, and delete the edge of $M$ containing $z$.
Note that, in this last case, $M(z)\in B_2$ becomes an endpoint in the resulting graph. 

Let $G_4$ be the resulting graph.
We claim that the following properties are satisfied:
\begin{enumerate}[label=(D\arabic*)]
    \item\label{equa:sizeboundsStep4} $|S|,|D|,|K|\leq5\log^2n$.
    \item\label{item:D1} $P\subseteq G_4$. Furthermore, the component of $G_4$ containing $P$ is a path.
    \item\label{item:D2} $G_4$ is the union of at most $4\log^2n$ cycles and $(1+o(1))|R|/2$ paths, all vertex-disjoint.
    \item\label{alg1} For each cycle $\cC\subseteq G_4$ we have that $|V(\cC)\cap B_1|=|V(\cC)\cap B_2|\pm5\log^2n$.
    \item\label{item:D3} Each of the paths of $G_4$ has both endpoints in $B_2\cup R$.
    \item\label{item:D4} All edges in $G_4\setminus G_1'$ are contained in $V(P)$ or have one endpoint in $B_1$ and the other in $C_2$.
    \item\label{item:D5} All subpaths of $G_4\setminus P$ have at most two consecutive vertices in $R\cup B_2\cup C_2$ or in $B_1\cup C_1$.
\end{enumerate}
Indeed, the sizes of $S$, $D$ and $K$ may increase by at most $14$ in this step, so \ref{equa:sizeboundsStep4} follows from \ref{equa:sizeboundsStep3}.
\ref{item:D1} follows from the construction, since no edges incident to any vertex of $P$ have been removed (this follows by the choice of the vertices $z$, \ref{item:B12} and \ref{item:B13}).
Note that $G_4$ is a union of paths and cycles since, again, we have made sure that every vertex has degree at most $2$.
Furthermore, even though the number of paths and cycles may have increased with respect to $G_3$, by \cref{remark:newcycles} we have that they increase by at most $2$.
Therefore, \ref{item:D2} follows from \ref{item:C2} and \eqref{equa:Rsize}.
\ref{alg1} follows by combining the bound on $|K|$ in \ref{equa:sizeboundsStep4} with \ref{item:A3}\COMMENT{Recall that, after Step~\hyperlink{Step1}{1}, all cycles alternate between edges of $G$ and edges of $M$ \ref{item:A3}. 
    It is the edges of $M$ which guarantee that $|V(\cC)\cap B_1|=|V(\cC)\cap B_2|$: if a vertex $v\in B_1$ lies in $\cC$, so must its neighbour in $B_2$ (and vice versa).
    Now, however, some of these edges have been removed, but we can bound the number of such edges by $|K|\leq5\log^2n$.
    Thus, for every vertex $v\in B_1$ which lies in $\cC$ (but at most $|K|$ of them), its corresponding vertex in $B_2$ $M(v)$ must also lie in $\cC$ (and the same argument can be written vice versa).
    This gives the desired conclusion.}.
\ref{item:D3} holds by construction and \ref{item:C3}. 
Note, in particular, that the choices of the vertices $z$ together with \ref{item:C2} guarantee that we do not create any new endpoints in $C_2$.
\ref{item:D4} follows from \ref{item:C9} and the fact that the only edges we may have added in this step have one endpoint in $B_1$ and the other in $C_2$.
This in turn, together with \ref{item:C4}, implies that \ref{item:D5} holds.

\vspace{1em}

\hypertarget{Step5}{\textbf{Step 5.}}
Now we want to obtain a graph $G_5$ which satisfies the following properties:
\begin{enumerate}[label=(E\arabic*)]
    \item\label{item:E1} $P\subseteq G_5$.
    \item\label{item:E2} $G_5$ is the union of at most $(1+o(1))|R|/2$ paths, all vertex-disjoint.
    \item\label{item:E3} Each of the paths of $G_5$ has both endpoints in $B_2\cup R$.
\end{enumerate}
In order to achieve this, we must alter each cycle in $G_4$.
Crucially, we will show that, with these alterations, we do not create any new cycles.
We claim that, throughout the coming process, we may assume that the following properties hold:
\begin{enumerate}[label=(E\arabic*), start=4]
    \item\label{equa:sizeboundsStep5} $|S|,|D|,|K|\leq105\log^2n$.
    \item\label{item:E6} $P\subseteq G_4$. Furthermore, the component of $G_4$ containing $P$ is a path.
    \item\label{item:E5} All subpaths of $G_4\setminus P$ have at most two consecutive vertices in $R\cup B_2\cup C_2$ or in $B_1\cup C_1$.
\end{enumerate}
Note that, before the process starts, \ref{equa:sizeboundsStep5}, \ref{item:E6} and \ref{item:E5} hold by \ref{equa:sizeboundsStep4}, \ref{item:D1} and \ref{item:D5}, respectively.

We proceed by iteratively choosing a cycle $\cC\subseteq G_4$\COMMENT{Note that we may destroy more cycles than just the one we choose.} and considering the following cases.
\begin{enumerate}[label=\arabic*.]
    \item \hypertarget{step5.1}{If} $\cC$ has length at most $25$, we delete all its edges and add all its vertices to $S$.
    \item \hypertarget{step5.2}{Otherwise}, assume $\cC$ contains two vertices $x,y\in B_2$ with $\dist_{\cC}(x,y)\leq11$ and let $Q\subseteq\cC$ be an $(x,y)$-path of length $\dist_{\cC}(x,y)$.
    Then, remove all edges of $Q$ from $G_4$ and add all its internal vertices to $S$.
    \item \hypertarget{step5.3}{Otherwise}, since all cycles in $G_4$ are disjoint from $R$ (as follows from \ref{item:A2}, \ref{item:D1}, \ref{item:D4} and the fact that no new cycles are created throughout this step), by \ref{item:E5}, we may apply \cref{remark:trivial} to conclude that $\cC$ must contain three vertices $x,y,z\in C_2$ such that $\dist_\cC(x,y)\leq3$ and $\dist_\cC(y,z)\leq3$.
    
    \begin{claim}\label{Claim3}
    There exist distinct $v_1,v_2\in\{x,y,z\}$ such that the following holds.
    
    Let $Q$ be the shortest $(v_1,v_2)$-path in $G_4$.
    For each $i\in[2]$, there exists $w_i\in(N_H(v_i)\cap B_1)\setminus(K\cup N_\cC(v_i))$ (with $w_1\neq w_2$) such that, if we let $e_i$ be the edge of $M$ containing $w_i$, then $w_1$, $w_2$ and $v_1$ each lie in a different component of $G_4\setminus(Q\cup\{e_1,e_2\})$.
    \end{claim}
    
    \begin{claimproof}
    Recall that any two vertices $x',y'\in V(\cC)\cap B_2$ satisfy that $\dist_\cC(x',y')\geq 12$.
    By \ref{item:E5} and \cref{remark:trivial}, this means that $|V(\cC)\cap B_2|\leq|V(\cC)\cap C_2|/3$\COMMENT{We have that at most one in every four vertices in $B_2\cup C_2$ lies in $B_2$, that is,
    \[|V(\cC)\cap B_2|\leq|V(\cC)\cap(B_2\cup C_2)|/4=|V(\cC)\cap B_2|/4+|V(\cC)\cap C_2|/4,\]
    and the claim follows by reordering.}.
    Therefore, by \eqref{equa:gamma1bound} we have that $|V(\cC)\cap B_2|\leq\gamma_1n/3\leq29n/1000$\COMMENT{$|V(\cC)\cap B_2|\leq|V(\cC)\cap C_2|/3\leq|C_2|/3=\gamma_1n/3\leq(1/2-\gamma_2)n/3\leq(3/2-\sqrt{2})n/3\approx 0.028595...n$.}.
    Then, by \ref{alg1}, we conclude that
    \begin{equation}\label{equa:B2intersbound}
    |V(\cC)\cap B_1|\leq3n/100.
    \end{equation}
    
    Aiming for a contradiction, let us assume that the statement does not hold, that is, for each pair $v_1,v_2\in\{x,y,z\}$ we have that, for every $w_1\in(N_H(v_1)\cap B_1)\setminus(K\cup N_\cC(v_1))$ and every $w_2\in(N_H(v_2)\cap B_1)\setminus(K\cup N_\cC(v_2))$ with $w_2\neq w_1$, if for each $i\in[2]$ we let $e_i$ be the edge of $M$ containing $w_i$ (which, recall, must lie in $G_4$ by the definition of $K$), then at least two of the vertices $w_1$, $w_2$ and $v_1$ lie in the same component of $G_4\setminus(Q\cup\{e_1,e_2\})$.
    Let $X\coloneqq(N_H(x)\cap B_1)\setminus(K\cup V(\cC))$, $Y\coloneqq(N_H(y)\cap B_1)\setminus(K\cup V(\cC))$ and $Z\coloneqq(N_H(z)\cap B_1)\setminus(K\cup V(\cC))$.
    In particular, every vertex in $X\cup Y\cup Z$ lies in an edge of $M$ in $G_4$.
    
    Consider $x$ and $y$, and let $x_1\in X$ and $y_1\in Y$ be distinct (such vertices exist by \eqref{equa:CedgestoB}, \ref{equa:sizeboundsStep5}, and \eqref{equa:B2intersbound}).
    Let $Q_{xy}$ be the shortest $(x,y)$-path in $\cC$, and let $e_x$ and $e_y$ be the edges of $M$ which contain $x_1$ and $y_1$, respectively.
    Our choice of $x_1$ and $y_1$ guarantees that, in $G_4\setminus(Q_{xy}\cup\{e_x,e_y\})$, they lie in a different component than $x$ (and $y$), so $x_1$ and $y_1$ must lie in the same component.
    Let $F$ be the component of $G_4$ containing $x_1$ and $y_1$.
    By fixing $x_1$, if any choice of $y_2\in Y\setminus\{x_1\}$ lies in a component of~$G_4$ different from~$F$, we would reach a contradiction, so we must have that $Y\subseteq V(F)$, and similarly we have $X\subseteq V(F)$.
    
    Assume first that $F$ is a path, and let $u_F$ and $v_F$ be its endpoints.
    Given any pair of vertices $a,b\in V(F)$, we write that $a<_Fb$ if a traversal of $F$ starting at $u_F$ reaches $a$ before~$b$.
    Now assume that $x_1,x_2\in X$ and $y_1\in Y$ satisfy that $x_1<_Fy_1<_Fx_2$.
    Then, upon deleting the edge of $M$ containing $y_1$, it cannot be in the same component as both $x_1$ and $x_2$, so we would reach a contradiction.
    Thus, we must have that $x_1<_Fy_1$ for all $x_1\in X$ and $y_1\in Y\setminus\{x_1\}$, or $y_1<_Fx_1$ for all $x_1\in X$ and $y_1\in Y\setminus\{x_1\}$. 
    In particular, this implies that $|X\cap Y|\leq1$.
    One can similarly show that, if $F$ is a cycle, then $|X\cap Y|\leq2$\COMMENT{This is equally easy to show, but a bit harder to write down properly.}.
    
    By following the same arguments as above when considering $y$ and $z$, and $x$ and $z$, we conclude that $Z\subseteq V(F)$ and that $|Y\cap Z|\leq2$ and $|X\cap Z|\leq2$.
    Combining this with \eqref{equa:CedgestoB}, \ref{equa:sizeboundsStep5}, and \eqref{equa:B2intersbound}, we have that
    $|X\cup Y\cup Z|\geq101n/200$\COMMENT{$|X\cup Y\cup Z|\geq|X|+|Y|+|Z|-6\geq3(n/5-3n/100-55\log^2n)-6=0.51n-165\log^2n-6\geq0.505n$.}.
    But $X\cup Y\cup Z\subseteq B_1$ and $|B_1|\leq n/2$, the final contradiction.
    \end{claimproof}
    
    Let $v_1,v_2,w_1,w_2,e_1,e_2$ be given by \cref{Claim3}.
    Delete $E(Q)$, $e_1$ and $e_2$ from $G_4$, add $v_1w_1$ and $v_2w_2$ to $G_4$, and add all internal vertices of $Q$ to $S$.
\end{enumerate}

Observe that, in cases~\hyperlink{step5.1}{1} and~\hyperlink{step5.2}{2}, we trivially cannot create any new cycles, since no edges are added to the graph.
By \cref{remark:newcycles2}, we are also guaranteed that we do not create a new cycle in case~\hyperlink{step5.3}{3}.
This, together with \ref{item:D2}, means that the process described above is repeated at most $4\log^2n$ times.
Since in each iteration we delete at most $25$ edges (with the bound coming from case~\hyperlink{step5.1}{1}), \ref{equa:sizeboundsStep5} follows from \ref{equa:sizeboundsStep4}.
The fact that $P\subseteq G_4$ holds since, by \ref{item:B12}, \ref{item:B13}, the fact that $P$ is not contained in a cycle and the choices throughout the process, no edges incident to $P$ are added nor deleted.
The fact that no new cycles are created then implies that \ref{item:E6} must hold throughout.
Finally, \ref{item:E5} follows from the fact that all added edges have one endpoint in $B_1$ and the other in $C_2$.
All three properties must also hold after the process is finished.

Let $G_5$ be the graph resulting from the process above.
\ref{item:E1} follows directly from \ref{item:E6}.
\ref{item:E2} follows from \ref{item:D2}, \eqref{equa:Rsize}, \ref{equa:sizeboundsStep5} and since the increase in the number of paths in each iteration of the above process is clearly bounded by~$3$.
Finally, \ref{item:E3} holds by \ref{item:D3} and the construction (in particular, note that the choice of $w_1$ and $w_2$ guarantees that no endpoint is created in $C_2$).

\vspace{1em}

\hypertarget{Step6}{\textbf{Step 6.}}
Recall that, together, the paths described in \ref{item:E3} cover all vertices of $V(H)\setminus S$.
We are now going to iteratively combine the paths which conform $G_5$ into a single path with the same vertex set.
We will later turn this path into a cycle and absorb all vertices of~$S$ into it.
For simplicity of notation, from now on we update $G_5$ as well as~$D$ in each step; $S$ and $K$, however, are no longer updated.

To be more precise, our aim is to obtain a graph $G_6$ which satisfies the following properties:
\begin{enumerate}[label=(F\arabic*)]
    \item\label{item:F1} $P\subseteq G_6$.
    \item\label{item:F2} $G_6$ consists of a unique path on vertex set $V(H)\setminus S$.
    \item\label{item:F3} The endpoints of the path of $G_6$ lie in $B_2\cup R$.
\end{enumerate}
To achieve this, we will follow a process, each iteration of which reduces the number of components of $G_5$ by one.
We claim that the following properties hold throughout:
\begin{enumerate}[label=(F\arabic*),start=4]
    \item\label{equa:Dsize} $|D|\leq(1+o(1))|R|/2$.
    \item\label{item:F5} $G_5$ is a union of vertex-disjoint paths.
    \item\label{item:F6} The endpoints of all paths of $G_5$ lie in $B_2\cup R$.
\end{enumerate}
Observe that, before the process starts, \ref{equa:Dsize} follows from \ref{equa:sizeboundsStep5} and \eqref{equa:Rsize}, \ref{item:F5} holds by \ref{item:E2}, and \ref{item:F6} follows from \ref{item:E3}.

We proceed as follows.
While $G_5$ contains at least two paths, choose any such two paths $P_1$ and~$P_2$, and let $x$ be an endpoint of $P_1$, and $y$ be an endpoint of $P_2$.
In particular, by \ref{item:F6}, $x,y\in B_2\cup R$.
Choose some edge $e=zz'\in E_G(N_H(x)\cap N_H(y)\cap B_1)\setminus D$ (which must exist by \ref{item:PMup3}, \eqref{equa:tightbound} and \ref{equa:Dsize}).
\begin{itemize}
    \item If $e\notin E(P_1)\cup E(P_2)$, add the edges $xz$ and $yz'$ to $G_5$ and remove $zz'$.
    \item Otherwise, suppose $e\in E(P_1)$ and that $\dist_{P_1}(x,z)<\dist_{P_1}(x,z')$ (the other cases are similar).
    Then, remove $zz'$ from $G_5$ and add $xz'$ and $yz$.
\end{itemize}
In both cases, we clearly reduce the number of paths by one\COMMENT{In the first case, we do not combine the two chosen paths into one, but we split a third path into two and combine each `half' with each of the chosen paths. In the second case, we combine both chosen paths into a single path.}.
Furthermore, in each step we delete exactly one edge from $G_5$.
This, together with \ref{item:E2} and \ref{equa:sizeboundsStep5}, guarantees that \ref{equa:Dsize} holds throughout (and, in particular, this implies the process can indeed be carried out).
\ref{item:F5} follows by construction, as does \ref{item:F6}, since no new endpoints are created throughout.

Let $G_6$ be the graph resulting from the process so far.
As the process ends, \ref{item:F2} and \ref{item:F3} follow by construction.
Finally, \ref{item:F1} holds since, in all cases above, no edges incident to $V(P)$ are removed or added to the graph, as guaranteed by \ref{item:P2}.

\vspace{1em}

We can now complete the proof.
Let the endpoints of the unique component of $G_6$ (see \ref{item:F2}) be $x$ and $y$.
By \ref{item:F3} we have $x,y\in B_2\cup R$, so by \ref{item:PMup3}, \eqref{equa:tightbound} and \ref{equa:Dsize} we can take some edge $e=zz'\in E_G(N_H(x)\cap N_H(y)\cap B_1)\setminus D$.
Assume without loss of generality that $\dist_{P'}(x,z)<\dist_{P'}(x,z')$.
Then, by \ref{item:P2} and \ref{item:F1}, removing $zz'$ from $G_3$ and adding $xz'$ and $yz$ results in a cycle $\mathcal{C}$ with the same vertex set and such that $P\subseteq\mathcal{C}$.
Let $m\coloneqq|S|$, so $\mathcal{C}$ has length $n-m$.
We must now prove that there is a cycle of length $k$, for all $3\leq k\leq n$.
We split our analysis into three cases.

Assume first that $n-m\leq k\leq n$.
Consider a set $J\subseteq S$ with $|J|=k+m-n$.
For each $w\in J\setminus U$, choose a distinct edge $e_w=x_wy_w\in E_G(N_H(w))\setminus D$ (recall that for each $w\in U$ we have $w=u_i$ for some $i\in[t]$ and we already defined an edge $e_w\coloneqq x_iy_i$ in Step~\hyperlink{Step2}{2}). 
Note that \ref{item:PMup2}, \eqref{equa:tightbound2}, \ref{equa:sizeboundsStep5} and \ref{equa:Dsize} guarantee that there is a choice of edges as desired.
Then, for each $w\in J$, replace $e_w$ by the path $x_wwy_w$.
This clearly results in a cycle of the desired length.

Suppose next that $3\leq k\leq\alpha^2n/10$.
In such a case, consider any subpath $P'\subseteq \mathcal{C}$ of length $k-3$, and let its endpoints be $x$ and $y$\COMMENT{When $k=3$ we have a degenerate case with $x=y$.}.
Now choose any edge $zz'\in E_G(N_H(x),N_H(y))$ such that $z,z'\notin V(P')$ (the existence of such an edge follows by \ref{item:PMup2} and \ref{equa:sizeboundsStep5}).
Then, the union of $P'$ and the path $xzz'y$ forms a cycle of length $k$.

Finally, assume $\alpha^2n/10<k<n-m$.
Consider a subpath $P'\subseteq \mathcal{C}$ of length $k-3$ such that $P\subseteq P'$.
Let the endpoints of $P'$ be $x$ and $y$, respectively, and let $Z\coloneqq E_G(N_H(x),N_H(y))\setminus D$ (for notational purposes, when we write $zz'\in Z$ we assume that $z\in N_H(x)$ and $z'\in N_H(y)$).
Note that \ref{item:PMup2}, \eqref{equa:tightbound2} and \ref{equa:Dsize} imply that 
\begin{equation}\label{equa:Zbound}
    |Z|\geq \eta\alpha^2n.
\end{equation}
Recall that, by \ref{item:P2}, $Z$ and $P$ are vertex-disjoint.
We consider the following three cases.
\begin{enumerate}[label=\arabic*.]
    \item Assume that there exists $zz'\in Z$ such that $z,z'\notin V(P')$.
    Then, the union of $P'$ and the path $xzz'y$ forms a cycle of length $k$.
    \item Otherwise, let $Z'\coloneqq\{e\in Z:e\subseteq V(P')\}$, so we have that $Z'\subseteq E(P')$ and
    \begin{equation}\label{equa:Zbound2}
        |Z'|\geq|Z|-2\geq\eta\alpha^2n/2
    \end{equation}
    by \eqref{equa:Zbound}.
    Suppose there is an edge $zz'\in Z'$ such that $\dist_{P'}(x,z')<\dist_{P'}(x,z)$.
    If so, then $(P'\setminus\{zz'\})\cup\{xz,yz'\}$ is a cycle of length $k-2$ which contains $P$. 
    To obtain a cycle of length~$k$, replace $x_1y_1$ and $x_2y_2$ by the paths $x_1u_1y_1$ and $x_2v_2y_2$, respectively.
    \item Otherwise, $Z'\subseteq E(P')$ and all $zz'\in Z'$ satisfy that $\dist_{P'}(x,z)<\dist_{P'}(x,z')$.
    Note that all edges in $Z'$ are pairwise vertex-disjoint by definition.
    Choose two edges $zz',ww'\in Z'$ with $\dist_{P'}(x,z)<\dist_{P'}(x,w)$ which minimise $\dist_{P'}(z,w)$ over all possible pairs of edges.
    Let~$P''$ be the $(z',w)$-subpath of $P'$, and let $\ell\coloneqq\dist_{P'}(z,w)$.
    By an averaging argument using \eqref{equa:Zbound2}, it follows that $1\leq\ell\leq2|V(P')|/(\eta\alpha^2n)<2\eta^{-1}\alpha^{-2}<t$.
    In particular, this shows that $P\nsubseteq P''$, so we must have $P\subseteq P'\setminus P''$.
    Then, the graph $(P'\setminus E(P''))\cup\{xw,yz'\}$ is a cycle of length $k-\ell$ which contains $P$.
    In order to obtain a cycle of length $k$, replace $u_1,\ldots,u_{\ell}$ by $x_1u_1y_1,\ldots,x_{\ell}u_{\ell}y_{\ell}$.\qedhere
\end{enumerate}
\end{proof}


\section{Concluding remarks and open problems}\label{sect:concl}

Binomial random graphs and random regular graphs are perhaps the two most studied random graph models.
As we have mentioned in the introduction, many results about randomly perturbed graphs have been obtained for the binomial random graph; it thus seems natural to study analogous problems in random regular graphs. 
We believe it would be interesting to study graphs perturbed by a random graph with a fixed degree sequence as well.
Very recently, the first author has also considered Hamiltonicity of graphs perturbed by a random geometric graph~\cite{ED21}.
Of course, this study should also be extended to other graph properties.

As we have observed, the behaviour of the graph $H$ perturbed by $G_{n,d}$ when $d=1$ and $d=2$ is quite different.
When $d=1$, we have shown that, if $\delta(H)\geq\alpha n$ with $\alpha>\sqrt{2}-1$, then a.a.s.~$H\cup G_{n,1}$ is Hamiltonian, but the same is not necessarily true if $\alpha<\sqrt{2}-1$.
On the other hand, for $d=2$, we have shown that $H\cup G_{n,2}$ is a.a.s.~Hamiltonian for far sparser graphs $H$ ($\delta(H)=\omega(n^{3/4}(\log n)^{1/4})$ suffices).
We believe this is far from optimal and should be true for even sparser graphs $H$.
We thus propose the following question.

\begin{question}
What is the minimum $f=f(n)$ such that, for every $n$-vertex graph $H$ with $\delta(H)\geq f$, a.a.s.~$H\cup G_{n,2}$ is Hamiltonian?
\end{question}

The only lower bound we can provide for this question is of order $\log{n}$, which is very far from the upper bound given by \cref{thm:2factor}.
Indeed, consider an $n$-vertex complete unbalanced bipartite graph $H=(A,B,E)$ where $|A|=\log{n}/5$.
It follows by a standard concentration argument (in the proof of \cref{lem:fewcycles}\ref{item:fewcycles1}, one can see that the variables $X_i$ are actually independent, hence standard Chernoff bounds are applicable) that a.a.s.~$G_{n,2}[B]$ contains at least $\log{n}/2$ cycles\COMMENT{The number of cycles of $G_{n,2}$ concentrates around $\log n$. Now, each vertex in $A$ can mean that at most one of the cycles is not contained within $B$, thus at most $\log n/5$ cycles are not contained in $B$.}.
Upon conditioning on this event, it is easy to check that $H\cup G_{n,2}$ does not contain a Hamilton cycle. 

From an algorithmic perspective, by retracing our proofs of \cref{thm:2factor,,thm:matching,,thm:matching2} as well as \cref{lem:partition}, it easily follows that, given an $n$-vertex graph $H$ with $\delta(H)\geq\alpha n$ and any $d$-regular graph $G$ which satisfies the statements of \cref{lem:fewcycles,lem:triangsand3paths} (the latter with respect to the sets defined in each of the proofs, which can be checked in polynomial time), there is a polynomial-time algorithm that finds cycles of any given length in $H\cup G$.

One could consider a generalisation of the results we have obtained for random perfect matchings by considering random $F$-factors (where an $F$-factor is a union of vertex-disjoint copies of $F$ which together cover the vertex set), for some fixed graph $F$, assuming the necessary divisibility conditions.
In general, we believe the behaviour here will be similar to that of random perfect matchings: if~$G$ is a uniformly random $n$-vertex $F$-factor, there will exist a specific value $\alpha^*=\alpha^*(F)\in(0,1/2)$, independent of $n$, such that, for every $\epsilon>0$, the following hold:
\begin{itemize}
    \item for every $n$-vertex graph $H$ with $\delta(H)\geq(1+\epsilon)\alpha^*n$, a.a.s.~$H\cup G$ is Hamiltonian, and
    \item there exists some $n$-vertex graph $H$ with $\delta(H)\geq(1-\epsilon)\alpha^*n$ such that $H\cup G$ is not a.a.s. Hamiltonian.
\end{itemize} 
\Cref{thm:matching} asserts that $\alpha^*(K_2)=\sqrt{2}-1$.
We propose the following conjecture.

\begin{conjecture}
For all $r\geq2$, we have that $\alpha^*(K_r)$ is the unique real positive solution to the equation $x^r+rx-1=0$\COMMENT{This is a number smaller than $1/r$ and which tends to $1/r$ as $r$ goes to infinity.}\COMMENT{A silly little remark: while in this paper we have proved that the conjecture holds for $r=2$, observe that it also holds for $r=1$, in which case it is a relaxation of Dirac's theorem.}.
\end{conjecture}

The lower bound for the conjectured value of $\alpha^*(K_r)$ is given by the same extremal example as for perfect matchings, that is, a complete unbalanced bipartite graph.
Indeed, consider a complete bipartite graph with parts $A$ and $B$, where $|A|=\alpha n$ and $|B|=(1-\alpha)n$, and let $G$ be a uniformly random $n$-vertex $K_r$-factor.
We are going to estimate the number of cliques of each size in $G[B]$.
For each $s\in[r]$, let $X_s$ be the number of components of $G[B]$ which are isomorphic to $K_s$.
Fix a vertex $v\in B$.
For each $s\in[r]$, the probability that the component containing $v$ in $G[B]$ is isomorphic to $K_s$ is (roughly) given by\COMMENT{Since $v$ lies in a copy of $K_r$, we must see how many of the copies of $K_r$ containing $v$ (roughly $\inbinom{n}{r-1}$) contains exactly $s$ vertices in $B$, and this is given by the first term.}
\[\binom{(1-\alpha)n}{s-1}\binom{\alpha n}{r-s}\frac{1}{\binom{n}{r-1}}.\]
Thus, we have that\COMMENT{The $1/s$ term compensates for the overcounting (each $K_s$ is counted from each of its vertices).}
\[\mathbb{E}[X_s]\approx(1-\alpha)n\binom{(1-\alpha)n}{s-1}\binom{\alpha n}{r-s}\frac{1}{\binom{n}{r-1}}\frac{1}{s}\approx\frac{n}{r}\binom{r}{s}(1-\alpha)^s\alpha^{r-s}.\]

Now observe that, since $H$ is a complete bipartite graph, in building a longest cycle, each vertex of $A$ can `absorb' each of the components of $G[B]$.
The conclusion is that, if the number of such components is larger than the number of vertices in $A$, then a Hamilton cycle is impossible.
That is, a necessary condition for a Hamilton cycle would be that
\[\sum_{s=1}^rX_s\leq\alpha n.\]
If we consider the expectations (for the lower bound, Markov's inequality provides sufficient concentration), we have that
\[\sum_{s=1}^r\mathbb{E}[X_s]\approx\frac{n}{r}\sum_{s=1}^r\binom{r}{s}(1-\alpha)^s\alpha^{r-s}=\frac{n}{r}(1-\alpha^r),\]
so our necessary condition becomes 
\[\frac{n}{r}(1-\alpha^r)\leq\alpha n\iff \alpha^r+r\alpha-1\geq0.\]

We think the problem might be interesting for other instances of $F$ as well.

In a different direction, we also believe the problem could be interesting for hypergraphs.
In particular, following work of \citet{AGIR20}, it is known that for every integer $r\geq2$ there exists an (explicit) constant $\rho(r)$ such that a random $d$-regular $r$-uniform hypergraph $G_{n,d}^{(r)}$ a.a.s.~contains a loose Hamilton cycle if $d>\rho(r)$, and a.a.s.~does not contain such a cycle if $d\leq\rho(r)$.
We propose the following question.

\begin{question}
Let $r\geq3$ be an integer.
For each $d\leq\rho(r)$, for which values of $\alpha$ (possibly as a function of $n$) is it true that for any $n$-vertex $r$-uniform hypergraph $H$ with $\delta(H)\geq\alpha n^{r-1}$\COMMENT{This could be relaxed to codegree conditions.} we have that a.a.s.\ $H\cup G_{n,d}^{(r)}$ contains a loose Hamilton cycle?
\end{question}

\section*{Acknowledgement}

We would like to thank Padraig Condon for nice discussions at an early stage of this project.
We are also indebted to the anonymous referees for their helpful comments and suggestions.


\bibliographystyle{mystyle} 
\bibliography{regpert}


\end{document}